\documentclass[a4paper,leqno,11pt]{amsart}

\usepackage{amssymb}
\usepackage{amsthm}
\usepackage{tikz-cd}
\usepackage{mathtools}
\usepackage{framed}
\usepackage[mathscr]{eucal}
\usepackage{fullpage}
\usepackage{hyperref}
\usepackage{enumitem}
\usepackage{wasysym}
\usepackage{xcolor}
\usepackage{scalerel}
\usepackage{comment}
\usepackage[all]{xy}
\hypersetup{colorlinks=true,citecolor=blue,urlcolor =black,linkbordercolor={1 0 0}}
\mathtoolsset{centercolon}
\title{Equivalences via twisted hyperholomorphic sheaves from transverse Lagrangian fibrations}
\author{Moritz Hartlieb}
\address{Mathematisches Institut, Universität Bonn, Endenicher Allee 60, 53115
Bonn, Germany}
\email{hartlieb@math.uni-bonn.de}

\author{Saket Shah} 
\address{Department of Mathematics, University of Michigan, Ann Arbor, MI 48109, USA}
\email{sakets@umich.edu}
\date{\today}

\newcommand{\cO}{\mathcal{O}}
\newcommand{\cA}{\mathcal{A}}

\newcommand{\cC}{\mathcal{C}}
\newcommand{\cD}{\mathcal{D}}
\newcommand{\cE}{\mathcal{E}}
\newcommand{\cF}{\mathcal{F}}
\newcommand{\cG}{\mathcal{G}}
\newcommand{\cH}{\mathcal{H}}

\newcommand{\cL}{\mathcal{L}}

\newcommand{\cP}{\mathcal{P}}
\newcommand{\cQ}{\mathcal{Q}}

\newcommand{\cU}{\mathcal{U}}

\newcommand{\cW}{\mathcal{W}}
\newcommand{\cX}{\mathcal{X}}
\newcommand{\cY}{\mathcal{Y}}

\newcommand{\bbC}{\mathbb{C}}

\newcommand{\bbP}{\mathbb{P}}
\newcommand{\bbQ}{\mathbb{Q}}
\newcommand{\bbR}{\mathbb{R}}

\newcommand{\bbZ}{\mathbb{Z}}

\newcommand{\congpf}{\xymatrix@L=0.6ex@1@=15pt{\ar[r]^-\sim&}}

\newcommand{\fg}{\mathfrak{g}}

\newcommand{\fM}{\mathfrak{M}}

\newcommand{\Pic}{\operatorname{Pic}}
\newcommand{\Br}{\operatorname{Br}}

\newcommand{\Db}{\operatorname{D}^b}
\newcommand{\ch}{\operatorname{ch}}
\newcommand{\rank}{\operatorname{rank}}

\newcommand{\rH}{\mathrm{H}}
\newcommand{\rHt}{\widetilde\rH}
\newcommand{\Hprim}{\rH_{\mathrm{prim}}}
\newcommand{\lprim}{\mathrm{lprim}}
\newcommand{\Hlprim}{\rH_{\lprim}}
\newcommand{\Htr}{\rH_{\mathrm{tr}}}
\newcommand{\Sym}{\operatorname{Sym}}
\newcommand{\Sd}{\operatorname{S}_{[d]}}
\newcommand{\SO}{\operatorname{SO}}
\renewcommand{\IJ}{\operatorname{IJ}}
\renewcommand{\div}{\operatorname{div}}
\newcommand{\iso}{\xrightarrow{\sim}}
\newcommand{\OG}{\mathrm{OG}10}
\newcommand{\Kthreen}{\mathrm{K}3^{[n]}}

\usepackage[T2A]{fontenc}
\DeclareSymbolFont{cyrillic}{T2A}{cmr}{m}{n}
\DeclareMathSymbol{\Sha}{\mathalpha}{cyrillic}{216}

\newtheorem{theorem}{Theorem}[section]
\newtheorem{lemma}[theorem]{Lemma}
\newtheorem{proposition}[theorem]{Proposition}

\newtheorem{assumption}{Assumption}
\newtheorem{conjecture}[theorem]{Conjecture}
\newtheorem{corollary}[theorem]{Corollary}

\theoremstyle{definition}
\newtheorem{definition}[theorem]{Definition}
\newtheorem{construction}[theorem]{Construction}
\newtheorem{question}[theorem]{Question}
\newtheorem{remark}[theorem]{Remark}

\newtheorem{example}[theorem]{Example}

\newcommand{\marking}{\varphi}
\newcommand{\isocls}{\eta}
\newcommand{\cnstsign}{\xi}

\newcommand{\mbmconst}{N}

\usepackage[
  backend=biber,
  style=alphabetic,
  sorting=nyt,
  maxcitenames=50,
  maxnames=50
]{biblatex}

\renewbibmacro{in:}{}
\AtEveryBibitem{%
  \clearfield{issn} 
  \clearfield{doi} 
  \clearfield{isbn}
  \ifentrytype{online}{}{
    \clearfield{url}
  }
}

\DeclareSourcemap{
  \maps[datatype=bibtex]{
    \map{
      \step[typesource=arxiv,       typetarget=misc]
      \step[fieldset=entrysubtype, fieldvalue=arxiv]
    }
  }
}

\DeclareFieldFormat[misc]{title}{%
  \iffieldequalstr{entrysubtype}{arxiv}
    {\mkbibquote{#1}}
    {\mkbibemph{#1}}}

\setlist[enumerate]{label={\rm{(\roman*)}}}

\numberwithin{equation}{section}

\begin{document}

\begin{abstract}
    Following ideas of Kapustka--Kapustka \cite{kapustkaconstruction}, we use Lagrangian fibrations to construct twisted hyperholomorphic sheaves on products of hyperk\"ahler manifolds of $\Kthreen$- and $\OG$-type. As applications, we prove the Lefschetz standard conjecture and the D-equivalence conjecture for hyperk\"ahler manifolds of $\OG$-type. 
\end{abstract}

\maketitle

\tableofcontents

\section{Introduction}
Projective hyperk\"ahler manifolds have proved a fruitful testing ground for many open problems in algebraic geometry. Many special tools are available in this setting. By the work of Huybrechts and Verbitsky, the Hodge theory of their second cohomology is exceedingly well-understood and is closely related to their birational geometry \cite{basicresults, huybrechtstorelli, verbitsky}. Other aspects of their birational geometry (such as their nef and movable cones) have also been described by Markman and Amerik--Verbitsky \cite{markmanmonodromy,amerikverbitsky}. \par
One other aspect of the theory that has recently found applications for hyperk\"ahler manifolds of $\Kthreen$-type is Markman's hyperholomorphic sheaf \cite{markmanalgebraic}. Building on work of Verbitsky and Buskin \cite{verbitskyhyperholomorphic, buskin}, Markman shows that for a K3 surface $S$ and a (fine) moduli space $M$ parametrizing stable sheaves on $S$ which is also a K3 surface, there is a sheaf $\cU^{[n]}$ on $S^{[n]} \times M^{[n]}$ produced from the universal sheaf $\cU$ on $S \times M$ which has excellent properties. Most importantly, the sheaf $\cU^{[n]}$ deforms over a certain moduli space of pairs of hyperk\"ahler manifolds of $\Kthreen$-type. \par 
We mention a few applications of Markman's hyperholomorphic sheaves. Let us recall the following conjecture of Grothendieck. For any smooth projective variety $X$ and any ample divisor $L$ on $X$, we have the cup product $(L\cup - ) \in \mathrm{End}(\rH^*(X,\bbQ))$. This operator naturally extends to a $\mathfrak{sl}_2$-triple $(L \cup -,h, \Lambda)$, where $h|_{\rH^d(X,\bbQ)}$ is multiplication by $d$. 
\begin{conjecture}[{Lefschetz standard conjecture \cite[p.\ 196]{grothendiecklsc}}]\label{conj:lsc}
    The dual Lefschetz operator $\Lambda$ is induced by an algebraic cycle on $X \times X$. 
\end{conjecture}
This was proved for all $\Kthreen$-type hyperk\"ahler manifolds first by Charles--Markman \cite{k3nlsc}, but Markman observed in \cite{markmanalgebraic} that the existence of the hyperholomorphic sheaf $\cU^{[n]}$ (or rather, its deformations) gives a very short and different proof of the same result. \par
In the direction of derived categories, Bondal--Orlov and Kawamata separately conjectured the following \cite{bondalorlov,kawamata}.

\begin{conjecture}[D-equivalence conjecture]\label{conj:dequiv}
If $X$ and $X'$ are smooth projective birational Calabi--Yau varieties, then there is an equivalence of bounded derived categories
\[D^b(X) \simeq D^b(X').\]
\end{conjecture}
Relying on the existence of derived equivalences between hyperkähler manifolds of $\Kthreen$-type arising from Markman's hyperholomorphic sheaf, Maulik--Shen--Yin--Zhang have recently shown that Conjecture~\ref{conj:dequiv} holds for hyperkähler varieties of $\Kthreen$-type. \par 
The goal of this paper is to provide a new construction of twisted hyperholomorphic sheaves on certain products $X \times Y$ of hyperk\"ahler manifolds. Most interestingly, our construction holds not only in the $\Kthreen$ deformation type, but also for hyperk\"ahler manifolds of $\OG$-type. As applications, we show that Conjecture~\ref{conj:lsc} and Conjecture~\ref{conj:dequiv} hold for hyperkähler varieties of $\OG$-type, see Theorem~\ref{thm:lscintro} and Theorem~\ref{thm:dequiv_intro} below.

The key idea in our construction goes back to work of Kapustka--Kapustka \cite[Sections 6 and 7]{kapustkaconstruction}, and we outline the general strategy below.  \par 
Begin with a projective hyperk\"ahler manifold $Y$ of $\Kthreen$- or $\OG$-type which admits two transverse Lagrangian fibrations \[\begin{tikzcd}
            & Y \arrow[ld, "\pi_1"'] \arrow[rd, "\pi_2"] &             \\
\mathbb P^n &                                            & \mathbb P^n,
\end{tikzcd}\]
i.e., we assume that each fiber of $\pi_1$ intersects every fiber of $\pi_2$ in finitely many points. \par 
In good situations (which can be achieved via Hodge theory, see Construction \ref{construction:hklattice}), each Lagrangian fibration separately realizes $Y$ as a Tate--Shafarevich twist of a different compactified abelian scheme with all fibers integral. In the $\Kthreen$ case, the main ingredient for this is work of Markman \cite{markmanlagrangian}. In the $\OG$ case, we rely on results by Dutta--Mattei--Shinder \cite{ijtwists}, which characterize Tate--Shafarevich twists of the LSV fibration associated to a cubic fourfold. \par
As a result, we produce a triple of Lagrangian-fibered hyperk\"ahler manifolds of $\Kthreen$- or $\OG$-type
\begin{equation*}
    \begin{tikzcd}
        X \arrow[rd, "\pi_X"'] & & Y \arrow[ld, "\pi_1"] \arrow[rd, "\pi_2"'] & & Z, \arrow[ld, "\pi_Z"] \\
        & \bbP^n & & \bbP^n & 
    \end{tikzcd}
\end{equation*}
where $X$ and $Z$ are compactifications of abelian schemes such that $\pi_1$ (resp. $\pi_2$) realizes $Y$ as a Tate--Shafarevich twist of $X$ (resp. $Z$). \par 
Relying on work of Arinkin \cite{arinkin} on derived equivalences for compactified Jacobians, a twisted generalization by Bottini \cite{bottinitenfolds}, as well as analogues for LSV fibrations coming from recent work of Yu \cite{yulsv}, there are Brauer classes $\alpha_X \in \Br(X)$ and $\alpha_Z \in \Br(Z)$ along with twisted sheaves \[\cP_{X,Y} \in \Db(X\times Y,\mathrm{pr}_1^*\alpha_X), \quad \cP_{Y,Z} \in \Db(Y \times Z, \mathrm{pr}_2^*\alpha_Z).\]
We then let $\cQ \in \Db(X \times Z, \alpha_X\boxtimes \alpha_Z)$ be the Fourier--Mukai kernel of the composition of these functors. \par 
For certain hyperk\"ahler fourfolds arising from the double EPW sextic construction, Kapustka--Kapustka show that $\cQ$ is a hyperholomorphic sheaf, and they conjectured that the same should hold for a more general class of hyperkähler manifolds of $\Kthreen$-type \cite{kapustkaconstruction}. By relying on a number of technical results of Bottini \cite{bottinitenfolds}, we are able to prove the following generalization of the result of Kapustka--Kapustka: 
\begin{theorem}[Theorem \ref{thm:hyperholexistence}]\label{thm:hyperholexistenceintro}
Let $Y$ be a projective hyperkähler manifold of $\Kthreen$- or $\OG$-type. Then the object $\cQ \in \Db(X \times Z, \alpha_X\boxtimes \alpha_Z)$ constructed above is a twisted hyperholomorphic vector bundle on the product of the two hyperkähler manifolds $X$ and $Z$ of $\Kthreen$- or $\OG$-type.
\end{theorem}
As applications, we prove the following two theorems, separately giving analogues of the main theorems of \cite{k3nlsc} and \cite{msyz} in the $\OG$ case. 
\begin{theorem}[Theorem~\ref{thm:lsc}]\label{thm:lscintro}
   Conjecture~\ref{conj:lsc} holds for hyperk\"ahler varieties of $\OG$-type. 
\end{theorem}
Notably, this result generalizes \cite[Theorem 1.2]{lsvlefschetzstd} as well as \cite[Corollary 1.9]{ffzlsc} which prove Conjecture \ref{conj:lsc} for various lower-dimensional strata in the moduli space of hyperkähler manifolds of $\OG$-type.
\begin{theorem}[Theorem~\ref{thm:dequivalence}]\label{thm:dequiv_intro}
    Conjecture~\ref{conj:dequiv} holds for hyperkähler varieties of $\OG$-type.
\end{theorem}
One should observe that Theorem~\ref{thm:hyperholexistenceintro} provides a new construction of (twisted) hyperholomorphic vector bundles on products of two hyperk\"ahler manifolds even for those of $\Kthreen$-type. It would be interesting to compare this to the one previously considered. 
\begin{question}
    Is the twisted hyperholomorphic vector bundle $\cQ$ deformation equivalent along diagonal twistor lines to one of the vector bundles $\cU^{[n]}$ considered by Markman in \cite{markmanalgebraic}? 
\end{question}
\subsection{Outline of the paper}
We review the necessary background on hyperholomorphic sheaves, Poincar\'e sheaves, and the Hodge theory and birational geometry of hyperk\"ahler manifolds in Section \ref{sec:background}. In Section \ref{sec:lsvtwists}, we show that a general hyperk\"ahler manifold of $\OG$-type admitting a rational Lagrangian fibration is birational to a Tate--Shafarevich twist of the LSV fibration. Using the theory of extended Mukai vectors, we investigate the cohomological action of the Poincar\'e sheaf for Beauville--Mukai and LSV fibrations in Section \ref{sec:cohomFM}, where we also define the \textit{Lagrangian primitive cohomology} of such hyperk\"ahler manifolds. We combine these ingredients in Section \ref{sec:main_construction}, where we construct examples of products of hyperk\"ahler manifolds which admit this new hyperholomorphic sheaf. Finally we present the applications of this sheaf to hyperk\"ahler manifolds of $\OG$-type in Section \ref{sec:applications}.
\subsection{Acknowledgments}

We would like to thank Younghan Bae, Daniel Huybrechts, Alex Perry, Laura Pertusi and Weite Pi for their feedback on a preliminary version of this draft. Special thanks go to Daniel Huybrechts and the University of Bonn for hosting the second author during the preparation of this preprint. We would also like to thank Alessio Bottini, Joe Foster, James Hotchkiss, Eyal Markman, and Lisa Marquand for interesting and useful conversations. 

The first author is grateful for the support provided by the ERC Synergy Grant 854361 HyperK as well as the International Max Planck Research School on Moduli Spaces at the Max Planck Institute for Mathematics in Bonn. The second author is grateful for the support provided by the National Science Foundation under NSF grant DMS-2143271 and NSF RTG grant DMS-1840234, as well as by the University of Michigan under the Rackham Conference Travel Grant. 

\section{Background} \label{sec:background}

Throughout, by a \emph{hyperkähler manifold}, we mean a simply connected compact Kähler manifold $X$ such that $\rH^0(X, \Omega^2_X) = \mathbb C \omega$, where $\omega$ is a everywhere nondegenerate holomorphic $2$-form on $X$.

\subsection{Isometries on the cohomology of \texorpdfstring{hyperk\"ahler}{hyperkähler} manifolds} 
In this section, we recall a number of basic results on isometries of the rational cohomology of a hyperk\"ahler manifold which will be useful in the sequel. \par 
\subsubsection*{Hodge isometries from (twisted) derived equivalences}
Most of the cohomological isometries appearing in this paper will arise from twisted derived equivalences. One way to formulate this is by making use of twisted Chern characters, as in \cite[Section 1]{huybrechtsstellari}. \par 
In our setting, assume $X$ and $Y$ are smooth projective varieties and $\cE$ is a $(-\mathrm{pr}_1^*\alpha + \mathrm{pr}_2^*\beta)$-twisted sheaf on $X \times Y$ for $\alpha \in \Br(X)$, $\beta \in \Br(Y)$. If $\alpha$ and $\beta$ are topologically trivial, so that they admit B-field lifts $B_\alpha \in \rH^2(X,\bbQ)$, $B_\beta \in \rH^2(Y,\bbQ)$; then $B = -\mathrm{pr}_1^*B_\alpha + \mathrm{pr}_2^*B_\beta$ is a B-field lift for $-\mathrm{pr}_1^*\alpha + \mathrm{pr}_2^*\beta$. As in \cite{bottinitenfolds}, we take 
\begin{equation}
    \Phi_\cE^{\rH} \coloneq \mathrm{pr}_{2,*}(v^B(\cE)\cdot\mathrm{pr_1^*}(-)) \colon \rH^*(X,B_\alpha,\bbQ) \to \rH^*(Y,B_\beta,\bbQ).
\end{equation}
where $v^B(\cE) \coloneq \sqrt{\mathrm{td}(X\times Y)}\cdot\ch^B(\cE)$. Following \cite[Remark 1.3]{huybrechtsstellari}, we use $\rH^*(X,B_\alpha,\bbQ)$ to denote the \textit{rational twisted Hodge structure}, i.e., the underlying lattice $\rH^*(X,\bbQ)$ with Hodge structure pulled back along the isometry
\[\rH^*(X,\bbQ) \xrightarrow{\cdot\exp(B_\alpha)}\rH^*(X,\bbQ).\]
\subsubsection*{LLV-equivariant isometries}
Fix $X$ a smooth projective hyperk\"ahler manifold of dimension $2n$. By Looijenga--Lunts and Verbitsky \cite{looijengalunts, verbitskythesis}, there is a canonical Lie algebra $\fg(X) \subset \mathrm{End}(\rH^*(X,\bbQ))$ which is called the LLV Lie algebra generated by all $\mathfrak{sl}_2$-triples $(e_\omega,f_\omega,h_\omega)$ for $\omega \in H^2(X,\bbQ)$ satisfying hard Lefschetz. 
\begin{definition} \label{def:llvequiv}
    Let $\phi \colon \rH^*(X,\bbQ) \to \rH^*(Y,\bbQ)$ be an isometry. We say that $\phi$ is an \textit{LLV-equivariant} isometry if 
    \[\fg(Y) = \phi\fg(X)\phi^{-1} \subset \mathrm{End}(\rH^*(Y,\bbQ)).\]
    We call an isometry $\phi \colon \rH^*(X,B_\alpha,\bbQ) \to \rH^*(Y,B_\beta,\bbQ)$ of rational twisted Hodge structures LLV-equivariant if it is so on the underlying lattices. 
\end{definition}
The key observation is that LLV-equivariant isometries are those which induce maps of rational extended Mukai lattices
\[\rHt(X,\bbQ) \coloneq \bbQ e\oplus \rH^2(X,\bbQ) \oplus \bbQ f\]
where $e,f$ are the generators of a hyperbolic plane orthogonal to $\rH^2(X,\bbQ)$ such that $(e,f) = -1$, and where $\rH^2(X,\bbQ)$ is endowed with the Beauville--Bogomolov--Fujiki quadratic form. This lattice admits a weight 2 Hodge structure with $\rHt^{2,0}(X,\bbQ) = \rH^{2,0}(X,\bbQ)$, and is endowed with an obvious grading. \begin{remark}
    To prevent confusion with Brauer classes, our notation for extended Mukai lattices differs from \cite{taelman} and \cite{beckmann}, both of which use the notation $\alpha$ for $e$ and $\beta$ for $f$. 
\end{remark}\par 
The starting point which connects the LLV Lie algebra with the rational extended Mukai lattice is the observation of Verbitsky that, if we take 
\[\Sd \rHt(X,\bbQ) \coloneq \ker\left(\Sym^d \rHt(X,\bbQ) \to \Sym^{d-2}\rHt(X,\bbQ)\right)\]
to be the kernel of contraction by the quadratic form, then there is an isomorphism $\Psi \colon \mathrm{SH}^*(X,\bbQ) \to \Sd\rHt(X,\bbQ)$, where $\mathrm{SH}^*(X,\bbQ)$ is the LLV-invariant subalgebra of $\rH^*(X,\bbQ)$ generated by $\rH^2(X,\bbQ);$ moreover, $\mathrm{SH}^*(X,\bbQ) \subset \rH^*(X,\bbQ)$ is an irreducible $\fg(X)$-subrepresentation occurring with multiplicity 1. The following consequences were then proved by Taelman. 
\begin{theorem}[{\cite[Theorems B and C]{taelman}}] \label{thm:taelman}
    Suppose that $\phi \colon \rH^*(X,\bbQ) \to \rH^*(Y,\bbQ)$ is an LLV-equivariant isometry of two smooth hyperk\"ahler manifolds in one of the known deformation types. Then there is a unique isometry  up to sign $\widetilde\phi \colon \rHt(X,\bbQ) \to \rHt(Y,\bbQ)$ so that 
    the square \[
    \begin{tikzcd}
        \mathrm{SH}^*(X,\bbQ) \arrow[r, "\phi"] \arrow[d,"\Psi"'] & \mathrm{SH}^*(X,\bbQ) \arrow[d,"\Psi"] \\
        \Sd\rHt(X,\bbQ) \arrow[r, "\Sd\widetilde\phi"] & \Sd\rHt(Y,\bbQ)
    \end{tikzcd}\]
    commutes. Moreover, this assignment is functorial.
\end{theorem}
\begin{remark}
    To be more precise about the sign, when $n$ is odd the diagram above commutes on the nose. When $n$ is even, $\Sd\widetilde\phi = \Sd(-\widetilde\phi),$ so to make the assignment genuinely unique, one should fix an orientation on $\widetilde{H}(X,\bbQ)$ and multiply by $-1$ for orientation-reversing $\widetilde\phi.$
\end{remark}

\begin{example} \label{ex:taelmanllv}
    By another result of Taelman \cite[Theorem A]{taelman}, any (untwisted) derived equivalence of hyperk\"ahler manifolds \[\Phi_\cE \colon \Db(X) \to \Db(Y)\] induces an LLV-equivariant isometry on rational cohomology \[\Phi_\cE^\rH \colon \rH^*(X,\bbQ) \to \rH^*(Y,\bbQ)\] and therefore a map on extended Mukai lattices
    \[\Phi_{\cE}^{\rHt} \colon \rHt(X,\bbQ) \to \rHt(Y,\bbQ).\]
\end{example}
\begin{example} \label{ex:explambdallv}
    It also follows from the definition of the LLV Lie algebra that \[\exp(\lambda) \cdot (-) \colon \rH^*(X,\bbQ) \to \rH^*(X,\bbQ)\] is LLV-equivariant for $\lambda \in \rH^2(X,\bbQ)$ \cite[Proposition 3.2]{taelman}. The induced map \[\exp(\lambda) \cdot (-) \colon \rHt(X,\bbQ) \to \rHt(X,\bbQ)\] can be described as 
    \[\exp(\lambda) \cdot(re+\mu+sf) = r\left(e+\lambda + \frac{\lambda^2}{2}f\right) + (\mu + (\mu,\lambda)f) + sf.\]
\end{example}
In general, the induced isometry of rational Mukai lattices guaranteed by Theorem \ref{thm:taelman} satisfies several pleasing properties. 
\begin{proposition} \label{prop:extmukaiprop}
    Suppose $\phi \colon \rH^*(X,\bbQ) \to \rH^*(Y,\bbQ)$ is an LLV-equivariant isometry. 
    \begin{enumerate}[label=\rm{(\roman*)}]
        \item \label{proppart:extmukaihodge} \cite[Proposition 4.9]{taelman} When the isometry $\phi$ is Hodge, so is the induced isometry $\widetilde{\phi}$.
        \item \label{proppart:extmukaigrading} \cite[Lemma 3.4]{markmanalgebraic} When $\phi$ is degree-preserving (resp. reversing), so is $\widetilde\phi.$
    \end{enumerate}
\end{proposition}
Owing to Example \ref{ex:explambdallv} and Proposition \ref{prop:extmukaiprop}\ref{proppart:extmukaihodge}, when given the data of a B-field for a Brauer class, it is natural to define a \textit{twisted extended Mukai lattice}, cf. \cite[Section 2.2]{bottinitenfolds}. 
\begin{definition}
    Suppose $B \in \rH^2(X,\bbQ)$. The \textit{$B$-twisted extended Mukai lattice}
    \[\rHt(X,B,\bbQ)\]
    is defined as the underlying lattice $\rHt(X,\bbQ)$ with Hodge structure pulled back along \[\rHt(X,\bbQ) \xrightarrow{\exp(B)\cdot}\rHt(X,\bbQ).\]
\end{definition}
\subsubsection*{The \texorpdfstring{Poincar\'e}{Poincaré} sheaf isometry}
The Fourier--Mukai kernels we are interested in will arise from Cohen--Macaulay sheaves extending Poincar\'e bundles on compactifications of abelian schemes.

In the following, we suppose that either $X = \overline\Pic^0(S,L) \coloneqq \overline{\Pic}^0(\mathcal C / |L|) \to \bbP^n$ is a compactified Jacobian fibration of the universal curve $\mathcal C \to |L|$ of a linear system on a polarized K3 surface~$(S,L)$ satisfying that all members of $|L|$ are integral, or $X = \overline\IJ(C) \to \bbP^n$ is a compactified intermediate Jacobian fibration associated to a \emph{very good} cubic fourfold $C$, i.e., a smooth cubic fourfold not containing a plane, a cubic scroll, or a hyperplane section with a corank 3 singularity, see \cite[Definition 2]{verygoodcubics}. Note that, under this assumption, the fibers of the Lagrangian fibration $X \to \mathbb P^n$ are integral, cf.\ \cite[Theorem 1]{verygoodcubics}. Moreover, let $X_t \to \bbP^n$ be a projective Tate--Shafarevich twist of $X$. \par 
It follows from work of Arinkin \cite[Theorems A and C]{arinkin} that, in the case of the compactified Jacobian fibration, there exists a Poincar\'e sheaf $\cP$ on $X \times X$ inducing an equivalence \[\Db(X) \iso \Db(X).\]
The analogous statement for the compactified intermediate Jacobian fibration of a cubic fourfold  was recently proved by Yu \cite[Theorems 0.1 and 0.2]{yulsv}, using a descent argument and the fact that the fibers of the LSV fibration can be described as compactified Prym varieties.
We will need the following variants of these results, which accommodate Tate--Shafarevich twists.

\begin{theorem}[{\cite[Theorem 3.3]{bottinitenfolds} and \cite[Theorem 6.7]{yulsv}}] \label{thm:poincare}
    There is a natural Brauer class $\alpha_t \in \Br(X)$ and a $\mathrm{pr}_2^*\alpha_t$-twisted Cohen--Macaulay sheaf $\cP_{X_t, X}$ on $X_t \times_{\bbP^n} X \subset X_t \times X$ giving a twisted derived equivalence 
    \[\Phi_{\cP_{X_t, X}} \colon \Db(X_t) \iso \Db(X,\alpha_t).\]
\end{theorem}
The following proposition is an argument of Bottini; while he states the result for Tate--Shafarevich twists of a compactified Jacobian fibration $\overline\Pic^0(S,L)$, by the same proof it holds also for Tate--Shafarevich twists of compactified intermediate Jacobian fibrations.
\begin{proposition}[\cite{bottinitenfolds}, Proposition 3.9] \label{prop:extmukaipoincare}
Let $X_t \to \bbP^n$ be a Tate--Shafarevich twist of either a compactified Jacobian fibration or compactified intermediate Jacobian fibration (as in \cite{ijtwists}) $X \to \bbP^n$. There exists a B-field lift $B_t$ of $\alpha_t$ so that the induced isomorphism on rational cohomology 
\[\Phi^\rH_{\cP_t} \colon \rH^*(X_t,\bbQ) \to \rH^*(X,B_t,\bbQ)\]
decomposes as 
\[\rH^*(X_t,\bbQ) \xrightarrow{\rho} \rH^*(X,\bbQ) \xrightarrow{\Phi^\rH_{\cP}} \rH^*(X,\bbQ) = \rH^*(X,\bbQ,B_t),\]
where $\rho$ is parallel transport along the degenerate twistor line, $\Phi^\rH_\cP$ is the Fourier--Mukai transform associated to the untwisted Poincar\'e sheaf on $X\times X,$ and the last equality is the natural identification of $\rH^*(X,\bbQ)$ and $\rH^*(X,B_t,\bbQ)$ as lattices which does not preserve Hodge structures.
\end{proposition}
Applying results of Taelman (see Examples \ref{ex:taelmanllv} and \ref{ex:explambdallv}), one concludes: 
\begin{corollary} \label{cor:poincarellv}
    The isometry $\rH^*(X_t,\bbQ) \to \rH^*(X,B_t,\bbQ)$ is LLV-equivariant. In particular, there is an induced map of extended Mukai lattices 
    \[\Phi_{\cP_t}^{\rHt} \colon \rHt(X_t,\bbQ) \to \rHt(X,B_t,\bbQ).\]
\end{corollary}
\begin{remark} \label{rmk:transposermk}
    Since the normalized untwisted Poincar\'e sheaf $\cP$ is invariant under swapping the factors, Theorem \ref{thm:poincare} and the subsequent Proposition \ref{prop:extmukaipoincare} and Corollary \ref{cor:poincarellv} all have obvious analogues (with identical proofs) when the twisted Poincar\'e sheaf is treated as a kernel in the opposite direction. We will need this fact.
\end{remark}
The following observation of Bottini shows that the Fourier--Mukai transform induced by the Poincar\'e sheaf naturally produces twisted vector bundles from half-dimensional subspaces transverse to the Lagrangian fibers. The proof carries over verbatim to also handle the case of compactified intermediate Jacobian fibrations associated to very good cubic fourfolds.
\begin{proposition}[{\cite[Proposition 6.2]{bottinitenfolds}}] \label{prop:imageoffinlag}
    Let $Z$ be a subvariety of $X_t$ and $i \colon Z \hookrightarrow X_t$ the embedding. Let $G$ be a rank $1$ sheaf on $Z$ such that $i_* G$ is Cohen--Macaulay on $X_t$. Assume that the composition $Z \subset X_t \to \mathbb P^n$ is finite of degree $r$. Then 
    $E \coloneqq \Phi_{\cP_{X_t, X}}(i_*G)$ is a twisted locally free sheaf of rank $r$ on $X$. Moreover, if $i_*G$ is atomic, then $E$ is twisted atomic.
\end{proposition}
Recall that a polarization $H \in \mathrm{Amp}(X)_\bbR \subset \mathrm{Kah}(X)$ on $X$ is called $a$-\textit{suitable} if for any hyperplane orthogonal to a class $\lambda \in \rH^{1,1}(X,\bbZ)$ such that $-a \leq \lambda^2 < 0$, $H$ lies on the same half space as $\isocls_X =\pi_X^*\cO_{\bbP^n}(1)$. With respect to certain $a$-suitable polarizations, the twisted atomic vector bundles are in fact stable (and hence hyperholomorphic on $X$ by \cite[Proposition 2.6]{bottinitenfolds}, though we will not directly use this fact). The proof of the following theorem in \cite{bottinitenfolds} carries over verbatim to also handle the case of compactified intermediate Jacobian fibrations associated to very good cubic fourfolds.
\begin{lemma}[{\cite[Theorem 6.6]{bottinitenfolds}}] \label{lem:stabilitylemma}
    Under the hypotheses of Proposition \ref{prop:imageoffinlag}, the locally free twisted sheaf $\cE \coloneq \Phi_{\cP_{X_t,X}}(i_*G)$ is $\mu_H$-stable, where $H$ is an $a(\cE)$-suitable polarization for $a(\cE) \in \bbQ$ as defined for twisted modular vector bundles in \cite[Definition 6.4]{bottinitenfolds} and $$\mu_H(\cE) \coloneq \frac{c_1^B(\cE).H^{2n-1}}{\mathrm{rk}(\cE)}.$$
\end{lemma}

\subsection{Moduli spaces, periods, and twisted hyperholomorphic bundles} 
\subsubsection*{The moduli spaces \texorpdfstring{$\fM_\Lambda$}{M_λ} and \texorpdfstring{$\fM_\phi$}{M_ϕ}}
Fix an integral lattice $\Lambda.$ We will consider also the moduli space $\fM_\Lambda$ parametrizing pairs $(X,\marking_X)$ of a primitive hyperk\"ahler manifold with a marking
\[\marking_X \colon \rH^2(X,\bbZ) \iso \Lambda.\]
Fix now a rational isometry 
\[\phi \colon \Lambda_\bbQ \iso \Lambda_\bbQ.\] Following Buskin \cite[Section 4]{buskin} and Markman \cite[Section 5.2]{markmanalgebraic}, we also consider the moduli space of Hodge isometries $\fM_\phi \subset \fM_\Lambda \times \fM_\Lambda$ parametrizing tuples $(X,\marking_X,Y,\marking_Y)$ such that the rational isometry 
\[\marking_Y^{-1} \circ \phi \circ \marking_X \colon \rH^2(X,\bbQ) \iso \rH^2(Y,\bbQ)\]
sends some K\"ahler class of $X$ to some K\"ahler class of $Y$. \par 
Both $\fM_\Lambda$ and $\fM_\phi$ admit the structure of non-Hausdorff complex manifolds, which in the case of $\fM_\Lambda$ admits a simple description in terms of the period map 
\begin{align*}
    \cP \colon \fM_\Lambda &\to \Omega_\Lambda = \{[v] \in \bbP(\Lambda_\bbC) : v^2 = 0, (v,\bar{v}) > 0\} \\
    (X,\marking_X) &\mapsto \marking_X(H^{2,0}(X)).
\end{align*}
By Verbitsky's global Torelli theorem \cite{verbitsky}, see also \cite{huybrechtstorelli}, any two points $(X,\marking_X), (Y,\marking_Y)$ in $\fM_\Lambda$ are inseparable if and only if $\cP(X,\marking_X) = \cP(Y,\marking_Y)$ and the points are in the same component of $\fM_\Lambda$. \par 
One of the more interesting features of $\fM_\Lambda$ is the presence of two different types of special lines, which are called the twistor and degenerate twistor lines respectively. \par 
More precisely, if $X$ is a primitive hyperk\"ahler manifold with symplectic form $\sigma \in \rH^{2,0}(X)$, for any K\"ahler class $\omega \in \rH^{1,1}(X)$, the \textit{twistor family} is a natural deformation of complex structure $\cX \to \bbP^1$ where the base $\bbP^1 \hookrightarrow \fM_\Lambda$ is identified under the period map $\cP$ with \[\Omega_\Lambda \cap \bbP\marking_X(\langle\operatorname{Re} \sigma, \operatorname{Im}\sigma,\omega\rangle) \simeq \bbP^1.\] \par 
On the other hand, if $v \in \rH^{1,1}(X,\bbZ)$ is a nef isotropic class pulled back from the base of a Lagrangian fibration $X \to \bbP^n$, the \textit{degenerate twistor family} is a deformation of Lagrangian-fibered hyperk\"ahler manifolds $\cX \to \bbC$ where the period over $t \in \bbC \hookrightarrow \fM_\Lambda$ is given by $[\sigma + tv] \in \Omega_\Lambda.$ By \cite[Theorem 1.2]{abashevarogov} and \cite[Theorem 2.2.10]{abasheva}, see also \cite[Corollary 3.5]{huybrechtsbrilliant}, this family coincides with Tate--Shafarevich family produced by the standard regluing procedure. \par 
There is a similar notion to twistor lines on the level of the moduli space of Hodge isometries~$\fM_\phi$. Suppose $\omega_X \in \rH^2(X,\bbQ)$ is a K\"ahler class such that $\omega_Y \coloneq (\marking_Y^{-1} \circ \phi \circ \marking_X)(\omega_X)$ is also K\"ahler. Then there is a natural \textit{diagonal twistor line} $\bbP^1 \hookrightarrow \fM_\phi$ realizing the deformation $\cX \times_{\bbP^1} \cY \to \bbP^1$, where $\cX\to\bbP^1$ (resp.\ $\cY\to\bbP^1$) is the twistor deformation of $X$ (resp.\ $Y$) corresponding to the K\"ahler class $\omega_X$ (resp.\ $\omega_Y$). In fact, such lines cover the moduli space $\fM_\phi$. 
\begin{proposition}[{\cite[Lemma 5.15]{markmanalgebraic}}]
    Any two points $p = (X,\marking_X, Y,\marking_Y)$ and $q = (X',\marking_{X'},Y',\marking_{Y'})$ in a connected component $\fM_\phi^0 \subset \fM_\phi$ can be connected by a sequence of twistor lines \[p \in L_1, L_2, \dots, L_m \ni q\] with points 
    \[(X_i,\marking_{X_i},Y_i,\marking_{Y_i}) \in L_i \cap L_{i+1}\]
    for $1 \leq i \leq m-1.$ Moreover, we may assume that $\Pic(X_i) = 0$. 
\end{proposition}
Following \cite{markmanalgebraic}, we call such a path joining two points $p, q \in \fM_\phi$ a \textit{generic twistor path}. 
\subsubsection*{Twisted hyperholomorphic vector bundles on products}
In good cases, twisted vector bundles can be made to deform over generic twistor paths to any other point in the moduli space $\fM_\phi$ as a consequence of Verbitsky's theory of hyperholomorphic sheaves. Fix $X,Y$ primitive hyperk\"ahler manifolds, and suppose $\cE$ is a $(-\mathrm{pr}_1^*\alpha + \mathrm{pr}_2^*\beta)$-twisted vector bundle for $\alpha \in \Br(X),$ $\beta \in \Br(Y)$ which induces a derived equivalence 
\[\Phi_\cE  \colon \Db(X,\alpha) \to \Db(Y,\beta).\]
In particular, $\rank(\cE) > 0.$ We will proceed with a slightly reworded variant of the assumption \cite[Assumption $(\dagger)$]{bottinitenfolds}.
\begin{assumption} \label{ass:llvequiv}
    The Brauer classes $\alpha \in \Br(X)$ and $\beta \in \Br(Y)$ are topologically trivial, and respectively admit B-field lifts $B_\alpha \in \rH^2(X, \mathbb Q) $ and $B_\beta \in \rH^2(Y, \mathbb Q)$ such that the Hodge isometry 
    \[\Phi^{\rH}_\cE \colon \rH^*(X,B_\alpha,\bbQ) \to \rH^*(Y,B_\beta,\bbQ)\]
    is LLV-equivariant. 
\end{assumption}
By Corollary \ref{cor:poincarellv}, Assumption \ref{ass:llvequiv} is satisfied for the twisted Poincar\'e sheaves arising from Tate--Shafarevich we consider in this paper. \par 
Under this assumption, Theorem \ref{thm:taelman} and Proposition \ref{prop:extmukaiprop}\ref{proppart:extmukaihodge} ensures that there is a Hodge isometry 
\[\Phi^{\rHt}_\cE \colon \rHt(X,B_\alpha,\bbQ) \to \rHt(Y,B_\beta,\bbQ)\]
between twisted extended Mukai lattices. Let us define the class
\[\kappa(\cE) \coloneq \exp\left(-\frac{c_1^B(\cE)}{\rank(\cE)}\right)\cdot \ch^B(\cE),\]
which one can check is independent of the choice of B-field $B$. Note that \[-c_1^B(\cE)/\rank(\cE) \in \rH^2(X\times Y,\bbQ)\] is a B-field which differs from $B$ by a class in $\rH^{1,1}(X\times Y,\bbQ)$, so in particular $\kappa(\cE)$ is Hodge. Under the K\"unneth decomposition $\rH^2(X \times Y,\bbQ) \simeq \rH^2(X,\bbQ) \oplus \rH^2(Y,\bbQ)$ we may write 
\begin{equation} \label{eqn:newBfields}
    -c_1^B(\cE)/\rank(\cE) = -\mathrm{pr}_1^*B_\alpha' + \mathrm{pr}_2^*B_\beta'
\end{equation}
for two B-fields $B_\alpha'$ and $B_\beta'$ representing $\alpha$ and $\beta$ respectively. 
\begin{lemma}[{\cite[Lemma 4.1(1) and (2)]{markmanalgebraic}}] \label{lem:posrankdegreversing}
    Under Assumption \ref{ass:llvequiv}, we may associate to $\cE$ an LLV-equivariant isometry 
    \[\Psi_\cE \coloneq \mathrm{pr}_{2,*}\left(\sqrt{\mathrm{td}(X\times Y)}\cdot \kappa(\cE)\cdot\mathrm{pr}_1^*(-)\right) \colon \rH^*(X,\bbQ) \to \rH^*(Y,\bbQ),\] which fits into a commutative diagram 
    \begin{equation}
        \begin{tikzcd}
            \rH^*(X,\bbQ) \arrow[r,"\Psi_\cE"] \arrow[d,"\cdot\exp(-B_\alpha')"'] &  \rH^*(Y,\bbQ) \arrow[d,"\cdot\exp(-B_\beta')"] \\
            \rH^*(X,B_\alpha,\bbQ)\arrow[r,"\Phi_{\cE}^{\rHt}"] & \rH(Y,B_\beta,\bbQ)
        \end{tikzcd}
    \end{equation}
    and is both degree-reversing and Hodge. In particular, the induced map 
    \[\widetilde\Psi_\cE \colon \rHt(X,\bbQ) \to \rHt(Y,\bbQ)\]
    is also degree-reversing and Hodge. 
\end{lemma}
\begin{proof}
    Strictly speaking, this is proved in \cite[Lemma 4.1]{markmanalgebraic} when $\cE$ is an untwisted derived equivalence. However, under Assumption \ref{ass:llvequiv} the proof holds without change for twisted sheaves after replacing $c_1$ with $c_1^B$ and $\ch$ with $\ch^B$ wherever they appear. 
\end{proof}
The key insight of \cite{markmanalgebraic} was that, if the kernel $\cE$ is stable with respect to polarizations on $X \times Y$ which are suitably compatible with the isometry $\Psi_\cE$, then $\cE$ can be deformed over generic twistor paths. In our setting, this can be stated as follows.
\begin{theorem}[{\cite[Propositions 5.19 and 5.21]{markmanalgebraic}}] \label{thm:hyperholdef}
    Fix a constant $C = \pm 1$. Under Assumption \ref{ass:llvequiv}, let \[\psi \coloneq \widetilde\Psi_\cE|_{\rH^2(X,\bbQ)} \colon \rH^2(X,\bbQ) \to \rH^2(Y,\bbQ).\] For fixed markings $\marking_X$ for $X$ and $\marking_Y$ for $Y$, let $\phi \coloneq C\cdot \marking_Y \circ \psi\circ\marking_X^{-1}$, and let $\fM_\phi^0$ be the connected component of $\fM_\phi$ containing $(X,\marking_X,Y,\marking_Y)$. Suppose that there exists an open subcone $\cC_X \subset \mathrm{Kah}(X)$ such that for any $\omega_X \in \cC_X$, its image $\omega_Y \coloneq C\psi(\omega_X) \in \mathrm{Kah}(Y)$, and such that $\cE$ is slope-stable with respect to $\mathrm{pr}_1^*\omega_X + \mathrm{pr}_2^*\omega_Y$.  Then for any $(X',\marking_{X'},Y',\marking_{Y'}) \in \fM_\phi^0$ there exists a twisted vector bundle $\cE'$ deformation equivalent to $\cE$.
\end{theorem}
\begin{proof}
    After replacing Chern classes and Chern characters with their twisted variants and replacing the application of \cite[Theorem A]{taelman} in \cite[Proposition 5.21]{markmanalgebraic} with Assumption~\ref{ass:llvequiv}, the proof in the untwisted case goes through without change. Note that the results \cite[Corollary 6.12, Proposition 6.17, Lemma 7.2]{markmanbbclass} used as input to \cite[Proposition 5.19]{markmanalgebraic} are proved already in the twisted case. 
\end{proof}
Twisted vector bundles of the type arising in Theorem \ref{thm:hyperholdef} above are called \textit{hyperholomorphic}\footnote{Here we use the terminology loosely. Elsewhere in the literature, such a bundle would be called \textit{$\omega$-stable projectively hyperholomorphic} for $\omega = \mathrm{pr}_1^*\omega_X + \mathrm{pr}_2^*\omega_Y \in \mathrm{Kah}(X \times Y)$ for any $\omega_X \in \cC_X$ and $\omega_Y = C\cdot \psi(\omega_X)$. Since we are not fixing a particular K\"ahler class and since our vector bundle is already twisted, we choose to drop the additional verbiage.}. Combined with results of Kapustka--Kapustka \cite{kapustkaconstruction}, this shows that the existence of a derived equivalence from a twisted hyperholomorphic bundle at a single point of $\fM_\phi^0$ implies the existence of such a derived equivalence at any other point. 
\begin{theorem}[{\cite[Theorem 2.3]{kapustkaconstruction}}] \label{thm:hyperholequivalence}
    The vector bundle $\cE'$ of Theorem \ref{thm:hyperholdef} induces a twisted derived equivalence \[\Db(X',\alpha') \to \Db(Y',\beta')\] 
    for Brauer classes $\alpha' \in \Br(X')$ and $\beta' \in \Br(Y')$.
\end{theorem}
\begin{proof}
    Let $\mu_\omega$ denote the slope with respect to a K\"ahler class $\omega$. The proof in \cite[Theorem 2.3]{kapustkaconstruction} goes through essentially without change; the only subtlety of the proof where $\cE$ is twisted is showing that \[\cH om(\cE|_{\{p\}\times Y}, \cE|_{\{q\}\times Y}) \qquad \text{ and } \qquad \cH om(\cE|_{\{q\}\times Y},\cE|_{\{p\}\times Y})\] are $\mu_\omega$-polystable with respect to some $\omega$ when $\cE|_{\{p\}\times Y},\cE|_{\{q\}\times Y}$ are $\mu_\omega$-stable for any two points $p,q \in X$. Replacing the usage of \cite[Lemma 9.1]{markmanbbclass} with \cite[Proposition 6.6]{markmanbbclass}, one sees that the untwisted vector bundles
    \[\cE nd(\cE|_{\{p\}\times Y}),\qquad \cE nd(\cE|_{\{q\}\times Y})\]
    are $\mu_\omega$-polystable. Using now \cite[Lemma 9.1]{markmanbbclass} one shows that 
    \[\cE nd(\cE|_{\{p\}\times Y}) \otimes \cE nd(\cE|_{\{q\}\times Y}) \simeq \cH om(\cE|_{\{p\}\times Y}, \cE|_{\{q\}\times Y}) \otimes \cH om(\cE|_{\{q\}\times Y},\cE|_{\{p\}\times Y})\]
    is $\mu_\omega$-polystable. It is then enough to observe that for two indecomposable vector bundles $\cF, \cG$ such that $\cF \otimes \cG$ is $\mu_\omega$-polystable, $\cF$ must be $\mu_\omega$-stable. It is easy to see that for any destabilizing subsheaf $\cH \subset \cF$ will also destabilize $\cF \otimes \cG$, so $\cF$ is $\mu_\omega$-stable. Indeed, we may take $\cH \subset \cF$ with slope $\mu_{\omega}(\cH) = \mu_\omega(\cF)$, so that the $\mu_\omega$-polystability of $\cF \otimes \cG$ would force the inclusion $\cH \otimes \cG \subset \cF \otimes \cG$ to split, which in characteristic 0 implies the original extension splits, contradicting the indecomposability of $\cF$. \par 
    The remainder of the proof goes through without change. 
\end{proof}
The Brauer classes appearing in the derived equivalences of the deformations are determined by a twisted variant of a result of C\u ald\u araru.
\begin{proposition}[{\cite[Theorem 4.1]{caldararu}}] \label{prop:bfielddeformationprop}
    Consider a proper, smooth morphism $\cX \to S$ of complex analytic spaces. Suppose $\cE$ is a locally free $\alpha$-twisted sheaf on $\cX$ for some $\alpha \in \Br(\cX)$, and let $B \in \rH^2(\cX_0,\bbQ)$ be a B-field lift for $\alpha_0 \in \Br(\cX_0)$. Assume $S \ni 0$ is small enough that we are given an identification $\rH^i(\cX,\bbZ) \simeq \rH^i(\cX_0,\bbZ)$. Then 
    \[\alpha = [-c_1^B(\cE_0)/\rank(\cE_0)],\] 
    where $[-] \colon \rH^2(\cX,\bbQ) \to \Br(\cX)$ gives the underlying Brauer class of a B-field. 
\end{proposition}
\begin{proof}
    Let $m = \rank(\cE)$. It is enough to prove a twisted variant of \cite[Proposition 4.2]{caldararu}, i.e. that the ``topological twisting class'' $t(\cY_0/\cX_0) \in \rH^2(\cX_0,\bbZ/m\bbZ)$ of the $\bbP^{m-1}$-bundle $\cY_0 \coloneq \bbP(\cE_0)$ is $-c_1^B(\cE_0) \mod m$. \par 
    Observe first that the diagrams \cite[Equations (4.10) and (4.11)]{caldararu} still commute for sheaves of smooth functions (i.e., not necessarily holomorphic). Moreover, the topological twisting class map factors through the smooth variant:
    \[t \colon \rH^1(\cX_0,\mathrm{PGL}(m)) \to \rH^1(\cX_0,\mathrm{PGL}_{\mathrm{smooth}}(m)) \to \rH^2(\cX_0,\bbZ/n\bbZ).\] Recall that $c_1^B(\cE_0) \coloneq c_1(\cE_0 \otimes \cL^\vee)$ where $\cL$ is a topological $\alpha$-twisted line bundle constructed from the B-field $B$. It is enough then to observe that 
    \[[\bbP(\cE_0)] = [\bbP(\cE_0 \otimes \cL^\vee)] \in \rH^1(\cX_0,\mathrm{PGL}_{\mathrm{smooth}}(m)),\]
    which is immediate as the transition functions defining $\cE_0$ and $\cE_0 \otimes \cL^\vee$ differ by multiplication by a nonvanishing smooth function.
\end{proof}
\begin{corollary} \label{cor:bfielddeformationcor}
    The Brauer class $\alpha'$ (resp. $\beta'$) of Theorem \ref{thm:hyperholequivalence} is the image of the parallel transport of $B'_\alpha \in \rH^2(X,\bbQ)$ (resp.\ $B'_\beta \in \rH^2(Y,\bbQ)$) in $\Br(X')$ (resp.\ $\Br(Y')$) for $B_\alpha'$ and $B_\beta'$ as defined in \eqref{eqn:newBfields}.
\end{corollary}
\subsection{MBM classes on \texorpdfstring{hyperk\"ahler}{hyperkähler} manifolds}

By \cite{amerikverbitsky}, the birational geometry of hyperk\"ahler manifolds is controlled by certain special integral classes of negative square called \textit{MBM classes}. \par 
To phrase this more precisely, we recall that for $X$ a hyperk\"ahler manifold, the positive cone $\mathrm{Pos}(X) \subset \rH^{1, 1}(\bbR)$ is the connected component of $\{v \in \rH^{1,1}(X,\bbR) : (v,v) > 0\}$ containing Kähler classes. If we write $\mathrm{Kah}(X)$ for the cone generated by all K\"ahler classes on $X$, then there is an obvious inclusion $\mathrm{Kah}(X) \subset \mathrm{Pos}(X)$. In fact, the Kähler cone arises as a chamber of a natural wall-and-chamber decomposition induced by the MBM classes. 
\begin{theorem}[{\cite[Theorem 6.2]{amerikverbitsky}}] \label{thm:mbmcone}
    The K\"ahler cone $\mathrm{Kah}(X)$ is a connected component of 
    \[\mathrm{Pos}(X) \setminus \bigcup_{z \text{ MBM class}} z^\perp,\]
    and all other components arise as the K\"ahler cones of other birational models under parallel transport Hodge isometries. 
\end{theorem}
The union of the K\"ahler cones of all birational models of $X$ forms a convex cone $\mathrm{BirKah}(X)$ with wall-and-chamber structure coming from the MBM classes, whose chambers correspond precisely to the birational models of $X$.\footnote{Strictly speaking, one should take the closure of $\mathrm{BirKah}(X) \subset \mathrm{Pos}(X)$, as there are missing walls separating the different K\"ahler cones of the birational models.} We hence have inclusions
\[\mathrm{Kah}(X) \subset \mathrm{BirKah}(X) \subset \mathrm{Pos}(X).\]
MBM classes satisfy a very convenient finiteness condition.
\begin{theorem}[{\cite[Theorem 3.17]{mbmsquare}}] \label{thm:mbmbound}
    Suppose $X$ is a hyperk\"ahler manifold with second Betti number $b_2(X) \geq 5$. Then there exists a constant $\mbmconst > 0$ depending only the deformation type of $X$ such that 
    \[-\mbmconst < z^2 < 0\]
    for any primitive MBM class $z \in H^2(X, \mathbb Z)$ on $X$. 
\end{theorem}
Moreover, MBM classes are naturally deformation-invariant, in the following sense. 
\begin{proposition}[{\cite[Corollary 5.13]{amerikverbitsky}}]
    If $X$ is a hyperk\"ahler manifold and $z \in \rH^{1,1}(X,\bbZ)$ is an MBM class, then the parallel transport of $z$ to any deformation $X'$ of $X$ is an MBM class on $X'$ if and only if it remains of Hodge type on $X'$.
\end{proposition}
\section{Tate--Shafarevich twists of Laza--Sacc\texorpdfstring{\`a}{à}--Voisin systems} \label{sec:lsvtwists}
Let $Y$ be a hyperk\"ahler manifold. By \cite{abashevarogov,soldatenkovverbitsky,abasheva}, there is a good theory of Tate--Shafarevich twists of the Lagrangian fibration $\pi \colon Y \to \bbP^n$. As a consequence of their work, there is a natural family of hyperk\"ahler manifolds of the same deformation type \[\cY \to \bbC\]
called the \textit{Tate--Shafarevich family}, the fiber over $t \in \bbC$ of which is a Lagrangian-fibered hyperk\"ahler \[Y_t \to \bbP^n\] which is produced by regluing $Y \to \bbP^n$ along local vertical automorphisms over an open cover of the base \cite[Theorem A]{abasheva}. In the particular case where $Y = \overline\IJ(C)$ is a compactified intermediate Jacobian fibration of \cite{lsv}, these twists have been studied at length in \cite{ijtwists}.
\par 
Now let $X$ be any hyperk\"ahler manifold of $\OG$ type. In \cite{rapagnetta}, it is shown that the lattice $\rH^2(X,\bbZ)$ with the BBF form is always isometric to the lattice 
\[\Lambda_{\OG} \coloneq U^{\oplus 3} \oplus E_8(-1)^{\oplus 2} \oplus A_2(-1).\]

\begin{lemma}[{\cite[Lemma 2.1]{mongardionorati}}] \label{lem:isotropic}
    All primitive isotropic vector $v \in \Lambda_{\OG}$ are of divisibility $1$ and contained in a single orbit for the action of $O^+(\Lambda_{\OG})$.
\end{lemma}
As a consequence of Lemma \ref{lem:isotropic} and \cite[Theorem 4.6]{ijtwists}, given any primitive isotropic class $v \in \rH^2(X,\bbZ)$ it follows that the lattice $v^\perp/v$ is isometric to a Tate twist of the primitive cohomology of a cubic fourfold. Taking $v^\perp/v$ with the induced Hodge structure therefore naturally provides a point in the period domain of cubic fourfolds; since the image of the period map is known explicitly, as long as $v^\perp/v$ avoids the divisors $\cC_2,\cC_6$ we may associate to the pair $(X,v)$ a unique cubic fourfold. When the cubic fourfold is very good in the sense of Dutta--Marquand \cite[Definition 2]{verygoodcubics}, the manifold $X$ is itself related to the compactified intermediate Jacobian fibration of that cubic. The following characterization of Lagrangian fibrations on hyperkähler manifolds of $\OG$-type is analogous to the one for $\Kthreen$-type proved in \cite{markmanlagrangian}. Although it is already implicit from the description of Tate--Shafarevich twists of LSV fibrations as a divisor in the moduli space of $\OG$-type hyperk\"ahler manifolds in \cite[Section 1.3]{ijtwists}, we provide a proof for the benefit of the reader.

\begin{theorem} \label{thm:lagfibthm}
    Suppose that $X$ is a hyperk\"ahler manifold of $\OG$-type admitting a primitive isotropic class $v \in \rH^{1,1}(X,\bbZ)\subset\rH^2(X,\bbZ)$ such that there is a Hodge isometry $$v^\perp/v \simeq \Hprim^4(C,\bbZ)(1)$$ for $C$ a very good cubic fourfold. Then $X$ is birational to a Tate--Shafarevich twist $Y_t$ of the LSV fibration $Y = \overline\IJ(C).$ \par 
    If we assume in addition that $v$ is nef, then we may take the birational equivalence $\phi \colon X \dashrightarrow Y_t$ to commute with the induced Lagrangian fibrations on both $X$ and $Y_t$. 
\end{theorem}
For the Lagrangian fibration on $X$ induced by the nef class $v \in \rH^{1,1}(X,\bbZ)$, we mean the Lagrangian fibration $\pi \colon X \to \bbP^5$ such that $v = c_1(\pi^*\cO_{\bbP^5}(1))$ guaranteed by the SYZ conjecture, which is known in this case by \cite[Theorem 2.2]{mongardionorati}.
\begin{proof}[Proof of Theorem \ref{thm:lagfibthm}]
    Recall that the period of $\overline{\IJ}(C)$ for $C$ a very good cubic fourfold was determined in \cite[Corollary 4.7]{ijtwists} to coincide with
    \[U \oplus \Hprim^4(C,\bbZ)(1),\]
    where $U = \langle\isocls,\theta \rangle$ for $\isocls$ the isotropic class on $\overline{\IJ}(C)$ pulled back from $c_1(\cO_{\bbP^5}(1))$ on the base and $\theta$ a relative polarization for the Lagrangian fibration (which we may take to be isotropic as well). On the other hand, by \cite[Theorem 1.2]{abashevarogov}, the Tate--Shafarevich twists of $\overline\IJ(C)$ are realized by the degenerate twistor deformations of $\overline\IJ(C)$, which Hodge-theoretically correspond to replacing the symplectic form $\sigma_C$ of $\overline\IJ(C)$ by $\sigma_C + t\isocls$ for various $t \in \bbC,$ cf.\ \cite[Rem.\ 4.8]{ijtwists}. \par 
    By the birational Torelli theorem for manifolds of OG10-type, which follows from \cite[Theorem 5.4]{onoratithesis} and \cite[Corollary 6.3]{huybrechtstorelli}, it suffices to check that the period of $X$ coincides with the period of some degenerate twistor deformation of $\overline\IJ(C).$ Indeed, by Lemma \ref{lem:isotropic}, we have a (non-Hodge) isometry
    \[\rH^2(X,\bbZ) \simeq U \oplus v^\perp/v \simeq U \oplus \Hprim^4(C,\bbZ)(1),\]
    where $U = \langle v,w\rangle$ for some isotropic vector $w \in \rH^2(X,\bbZ)$ such that $(v,w) = 1.$ \par 
    Since $v$ is a Hodge class, it follows that the symplectic form $\sigma_X \in H^2(X,\bbC)$ of $X$ is necessarily orthogonal to $v$; then by the choice of the Hodge structure on $\Hprim^4(C,\bbZ)(1)$, under the isometry above, we have an identification 
    \[\sigma_X = (av,\sigma_C) \in U_\bbC \oplus \Hprim^4(C,\bbC)(1)\]
    for some $a \in \bbC.$ \par 
    Taking the natural isometry of lattices 
    \begin{equation}
        \rH^2(X,\bbZ) \iso \rH^2(\overline\IJ(C),\bbZ) \label{eqn:twistisometry}
    \end{equation}
    which is the identity on $\Hprim^4(C,\bbZ)(1)$, sending $v \mapsto \isocls$ and $w \mapsto \theta$, we find that the symplectic form of $\rH^2(X,\bbZ)$ is identified with $\sigma_C + a\isocls$. This proves the first claim. \par 
    For the second claim, we proceed as in \cite[Theorem 1.3, Step 2]{markmanlagrangian}. Suppose first that $X$ is non-projective. Since $\langle\operatorname{Re}\sigma,\operatorname{Im}\sigma_X\rangle$ span a positive-definite subspace, it follows that $\rH^2(X,\bbR) \cap \sigma_X^\perp$ has signature $(1,21)$, so by an easy lattice-theoretic argument any two linearly independent isotropic classes in $\Pic(X)$ cannot be orthogonal to each other. On the other hand, by Huybrechts' projectivity criterion \cite[Theorem 2]{basicresultserratum}, since $X$ is non-projective there is no class of positive square in $\Pic(X)$. It follows that there can be only one isotropic class in $\Pic(X)$ up to scaling. If we let $\phi \colon X \dashrightarrow Y$ be the birational map guaranteed by the first part, then in particular $\phi^*\isocls = \pm v$. Since both $\phi^*\isocls$ and $v$ are in the closure of the positive cone, it follows that $\phi^*\isocls = v$. \par 
    Now assume that $X$ is projective. Observe first that by choosing $w$ to lie in the closure of the positive cone as well, we may ensure that the isometry \eqref{eqn:twistisometry} is orientation-preserving and hence by \cite[Theorem 5.4]{onoratithesis} can be interpreted as a parallel-transport Hodge isometry 
    \[\rho \colon \rH^2(X,\bbZ) \iso \rH^2(Y_t,\bbZ).\]
    In particular, by \cite[Theorem 1.6]{markmanmonodromy} there exists a birational map $\phi \colon X \dashrightarrow Y_t$ so that the induced isometry
    \[ w = \phi^* \circ \rho \colon \rH^2(X,\bbZ) \iso \rH^2(X,\bbZ)\] 
    is a composition of reflections along prime exceptional divisors in $\Pic(X)$, i.e., prime divisors of negative BBF square. Since $v$ is nef, and $\phi^* \circ \rho(v) = \phi^*(\isocls)$ lies in the closure of the birational K\"ahler cone, by \cite[Proposition 5.6]{markmanmonodromy} both pair nonnegatively with every prime exceptional divisor. By a standard argument from the theory of hyperbolic reflection groups, since both $v$ and $w(v)$ pair nonnegatively with the prime exceptional divisors, we must have $w(v) = v$, cf. \cite[Theorem 1.3, Step 2.2]{markmanlagrangian}. \par 
    It follows then that $v = \phi^*\isocls,$ so the birational map $\phi$ respects the Lagrangian fibrations induced by $v$ and $\isocls$ respectively. 
\end{proof}
\begin{remark}
    By the surjectivity of the period map, there are still certainly Lagrangian-fibered hyperk\"ahler manifolds $X$ with primitive isotropic classes $v \in \rH^2(X,\bbZ)$ such that the period of $v^\perp/v$ is in $\cC_2$ or $\cC_6$, and it is interesting to ask how to geometrically access these hyperk\"ahler manifolds. \par 
    For $\cC_2$, we expect that these arise as desingularizations of singular Beauville--Mukai systems, i.e., $X = \widetilde M_H(0,2H,2)$ where $(S,H)$ is a K3 surface with $H^2 = 2$. In this case $X$ is a resolution of a moduli space of one-dimensional sheaves on curves, and taking the support gives a Lagrangian fibration on $X$, cf.\ \cite[Sec.\ 5.3]{klsv}. \par 
    For $\cC_6$, these should arise as the desingularized moduli spaces $X = \widetilde M_H(2,0,4)$ for $(S,H)$ a K3 surface with $H^2 = 6$. As was pointed out in \cite{hwangnagai}, $X$ will be birational to the intermediate Jacobian fibration for the singular cubic fourfold associated to $S$. 
\end{remark}

\section{Cohomological FM transforms from Poincar\texorpdfstring{\'e}{é} equivalences} \label{sec:cohomFM}
A key ingredient in our proofs will be an explicit understanding of the action of the Poincar\'e sheaf on the twisted extended Mukai lattice of the cohomology of certain compactified group schemes. \par 
Throughout this section, we assume that the $2n$-dimensional Lagrangian-fibered hyperk\"ahler 
\[\pi \colon X \to \bbP^n\]
is either the compactified Jacobian fibration $X = \overline\Pic^0(S,L)$ for a linear system of integral curves on a polarized K3 surface $(S,L)$ or is the LSV fibration $X = \overline\IJ(C)$ for a very good cubic fourfold $C$. In either of these cases, there is a distinguished rank 4 lattice of the extended Mukai lattice 
\[\langle e,f,\isocls,\theta\rangle \subset \rHt(X,\bbQ),\]
where $\isocls = c_1(\pi^*\cO_{\bbP^n}(1)) \in \rH^2(X,\bbZ)$ and $\theta \in \rH^2(X,\bbQ)$ is a generalized theta divisor for $\pi$ in the sense of \cite{generalizedbeauville}, i.e., a $\pi$-relatively ample $\bbQ$-divisor satisfying $\theta^2 = 0$ and $(\isocls,\theta) = 1$. In the $\Kthreen$ case, such a $\theta$ exists by \cite[Section 2.2]{generalizedbeauville}, while in the $\OG$ case we may even take $\theta \in \rH^2(X,\bbZ)$ integral by \cite[Lemma 3.5 and Proposition 3.6]{sacca}.

\begin{proposition}\label{prop:rank4sublattice}
    Let $\cP$ on $X \times X$ be the Poincar\'e sheaf in the $t = 0$ case of Theorem \ref{thm:poincare}. There exists a constant $\cnstsign = \pm 1$ (with $\cnstsign = 1$ when $n$ is odd) such that $\Phi^{\rHt}_\cP$ acts as 
    \begin{align*}
        f &\mapsto \cnstsign\isocls, & \theta &\mapsto \cnstsign\left(e - \frac{n+1}{2}\isocls\right),\\
        \isocls &\mapsto -\cnstsign f, & e &\mapsto -\cnstsign \left(\theta - \frac{n+1}{2}f\right).
    \end{align*}
\end{proposition}
\begin{proof}
    In the case of a compactified Jacobian, this is \cite[Proposition 10.4]{beckmann}, and is presented in the form above in \cite[Proposition 3.9]{generalizedbeauville}. \par 
    In the case of the LSV fibration $X = \overline\IJ(C)$ for a very good cubic fourfold $C$, the proof is essentially the same and uses the theory of extended Mukai vectors. We sketch the idea here. Note that as $n = 5$ is odd, so in this case we will expect $\cnstsign = 1$. \par 
    To check that $\Phi_\cP^{\rHt}(f) = \isocls$, observe that for any point supported on the zero section of the smooth locus of $X \to \bbP^n$, $\Phi_\cP(\cO_{\mathrm{pt}}) \simeq \pi^*{\cO}_{\mathrm{pt}}$, and as $v(\pi^*{\cO}_{\mathrm{pt}}) = \isocls^n$ the claim follows from \cite[Equation (4.10)]{beckmann}. \par 
    Next, observe that the inverse Fourier--Mukai functor is given by 
    \[\Phi_{\cP}^{-1} = \Phi_{\cP^\vee \otimes \mathrm{pr}_2^*\omega_{X/\bbP^n}[5]},\]
    where $\cP^\vee$ is the dual to $\cP$ on $X \times_{\bbP^n}X.$ In particular,
    \[(\Phi_\cP^{\rHt})^{-1} = -\Phi^{\rHt}_{\cP^\vee\otimes\mathrm{pr}_2^*\cO(6)} = -\exp(6\isocls) \cdot \Phi_{\cP^\vee}^{\rHt}.\]
    Since a similar argument shows that $\Phi^{\rHt}_{\cP^\vee}(f) = \isocls$, and since by \cite[Proposition 3.2]{taelman} the action of $B_{6\isocls} \coloneq \exp(6\isocls) \cdot \colon \rHt(X,\bbQ) \to \rHt(X,\bbQ)$ preserves $\isocls$, it follows that $(\Phi_{\cP}^{\rHt})^{-1}(f) = -\isocls.$ \par 
    To see the remaining classes, we consider the extended Mukai vector for the structure sheaf of the zero section\footnote{The zero section exists since by \cite[Theorem 1]{verygoodcubics} the LSV fibration has integral fibers as long as the cubic fourfold is very good. This was also observed in \cite[Corollary 4.8]{sacca} for a general cubic fourfold.} $\cO_Z$, which is sent under the functor $\Phi_{\cP}$ to $\cO_X$. Taking $v = \widetilde v(\cO_Z)$ as this extended Mukai vector, observe that we have
    \[\Phi^{\rHt}_{\cP}(v)= e + 2f\]
    and 
    \[\Phi^{\rHt}_{\cP^\vee}(v) = e + 2f\]
    by the normalization of the Poincar\'e sheaf \cite{yulsv} and \cite[Definition 4.1 and Equation (4.5)]{beckmann}. \par 
    When the cubic fourfold $C$ is very general, the only Hodge classes in $\rH^2(X,\bbQ)$ are $\isocls$ and $\theta$ by \cite[Lemma 3.5]{sacca}, so in that case we may write $v = ae + b\isocls + c\theta + df$ and determine coefficients using the equations above and the BBF form. Since the LSV construction works in families, the same holds for all very good cubic fourfolds. 
\end{proof}

As a consequence, it can be useful to separately treat this rank 4 lattice and its orthogonal complement. This motivates the following: 
\begin{definition} \label{def:lagprimcoh}
    Suppose $\pi \colon X \to \bbP^n$ is either $\overline\Pic^0(S,L)$ for a linear system of integral curves on a polarized K3 surface $(S,L)$ or $\overline\IJ(C)$ for a very good cubic fourfold $C$. The \textit{rational Lagrangian primitive cohomology} of $X$ is the rational sublattice 
    \[\Hlprim^2(X,\bbQ) \coloneq \langle \isocls,\theta\rangle^\perp \subset \rH^2(X,\bbQ),\]
    and the \textit{integral Lagrangian primitive cohomology} is 
    \[\Hlprim^2(X,\bbZ) \coloneq \Hlprim^2(X,\bbQ) \cap \rH^2(X,\bbZ).\]
\end{definition}
When $X$ is a very general hyperk\"ahler either of the form $\overline\Pic^0(S,L)$ or $\overline\IJ(C)$, the rational Hodge classes in $\rH^2(X,\bbQ)$ are generated by $\isocls$ and $\theta$, so in that case $\Hlprim^2(X,\bbZ) = \rH^2_\mathrm{tr}(X,\bbZ)$ coincides with the transcendental sublattice. 
\begin{example} \label{ex:lagprimk3n}
    If $X = \overline\Pic^0(S,L)$ for a linear system of integral curves on a polarized K3 surface $(S,L)$, then it follows from the usual description of $\rH^2(X,\bbZ)$ \cite{k3nhodgestructure} that \[\Hlprim^2(X,\bbZ) \simeq \Hprim^2(S,\bbZ),\] see also \cite[Lemma 4.1]{markmanlagrangian}.
\end{example}
\begin{example} \label{ex:lagprimog10}
    If $X = \overline\IJ(C)$ for a very good cubic fourfold $C$, then by \cite[Theorem 4.6]{ijtwists} 
    \[\Hlprim^2(X,\bbZ) \simeq \Hprim^4(C,\bbZ)(1).\]
\end{example}

Despite the fact that the cohomological action of the Poincar\'e equivalence may have denominators in the level of the rank 4 sublattice $\langle f,\isocls,\theta,e\rangle$, it is in fact extremely simple on the Lagrangian primitive cohomology, where we can prove the following. \par 

\begin{proposition} \label{prop:primitivepoincare}
    Suppose that $X = \overline\Pic^0(S,L)$ for a linear system of integral curves on a polarized K3 surface $(S,L)$ or $X = \overline\IJ(C)$ for a very good cubic fourfold $C$. Then the isometry $\Phi_{\cP}^{\rHt}$ restricts to an integral self isometry 
    \[\Hlprim^2(X,\bbZ) \iso \Hlprim^2(X,\bbZ)\]
    given by multiplication by $\pm 1$, where the sign depends only on the deformation type of $X$. 
\end{proposition}
\begin{proof}
    It follows immediately from Proposition \ref{prop:rank4sublattice} that $\Phi_{\cP}^{\rHt}$ restricts to a rational isometry 
    \[\Hlprim^2(X,\bbQ) \iso \Hlprim^2(X,\bbQ),\]
    so it suffices to show that $\Phi_{\cP}^{\rHt}(\lambda) = \pm \lambda$ for a consistent choice of sign whenever $\lambda \in \Hlprim^2(X,\bbZ).$ \par 
    Suppose first that $\Hlprim^2(X,\bbZ)$ has a unique Hodge class $\lambda$ up to scaling. As the BBF form on the algebraic part of $H^2(X, \mathbb Z)$ is non-degenerate, we have $(\lambda, \lambda) \neq 0$. Since $\Phi_{\cP}^{\rHt}$ is a Hodge isometry it follows immediately that $\Phi_{\cP}^{\rHt} = \pm \lambda.$ \par 
    Now let $\lambda \in \Hlprim^2(X,\bbZ)$ be any primitive class of large negative square (see also Remark~\ref{rem:largenegativesquare} for a more precise statement on what is meant by ``large negative square''). Since both the constructions of Beauville--Mukai and of LSV work in families, by using either Example \ref{ex:lagprimk3n} or Example \ref{ex:lagprimog10} and deforming to a very general point in a Noether--Lefschetz divisor of the moduli space of $L$-polarized K3 surfaces or of cubic fourfolds, we may assume that $\lambda$ is the unique integral Hodge class in $\Hlprim^2(X,\bbZ)$ and apply the argument above. \par 
    In fact, the sign must agree for all primitive classes of large negative square: suppose $\lambda_1,\lambda_2 \in \Hlprim^2(X,\bbZ)$ are any two primitive classes of large negative square and note that we can find $p, q \neq 0$ such that 
    \[p\lambda_1 + q\lambda_2\]
    will be primitive of large negative square. Then 
    \[ \pm(p\lambda_1 + q\lambda_2)= \Phi_{\cP}^{\rHt}(p\lambda_1 + q\lambda_2) = \pm p(\lambda_1) + \pm q(\lambda_2),\]
    which forces every sign to coincide. \par 
    Since by the following Lemma \ref{lem:primitivedifference} any vector $v \in \Hlprim^2(X,\bbZ)$ can be written as a difference of two primitive vectors of large negative square, the result follows. 
\end{proof}
\begin{lemma} \label{lem:primitivedifference}
    Let $\Lambda$ be any indefinite integral lattice. For any $v \in \Lambda$ and any $M > 0$, we can write $v = v_1 - v_2$ where $v_1,v_2 \in \Lambda$ are primitive integral classes with $v_1^2,v_2^2 < -M$.
\end{lemma}

\begin{proof}
    Since $\Lambda$ is indefinite, there is a primitive vector $w_1 \in \Lambda$ of square $w_1^2 < 0$. Then, extend $v_1$ to a basis $w_1, w_2$ of the saturation $\langle v, w_1 \rangle^{\mathrm{sat}} \subset M$ and write $v = aw_1 + b w_2$. Pick some prime $p \gg 0$ coprime to $a$ and set
    \[v_1 \coloneq (p+(k+1)a)w_1 + (a+b)w_2 \text{ and } v_2 \coloneqq (p+ka)v_1 +a v_2\]
    so that $v = w_1 - w_2$. Since $w_1^2 < 0$, we have $\max(v_1^2, v_2^2) < -M$ for $k \gg 0$. Moreover, by Dirichlet's theorem, there are infinitely many $k \in \mathbb N$ satisfying
    $$\gcd(p+(k+1)a, a+b) = 1 = \gcd(p+ka, a),$$
    which implies that $v_1$ and $v_2$ are primitive.
\end{proof}

\begin{remark}
    Since for a very general $X = \overline\Pic^0(S,L)$ as above, the Lagrangian primitive cohomology is just the transcendental cohomology, the integrality of the restriction $\Phi^{\rHt}_\cP|_{\Hlprim^2(X,\bbZ)}$ follows for a very general Beauville--Mukai system from \cite[Corollary 9.3]{beckmann} as a consequence of Beckmann's theory of integral $\Kthreen$ lattices. 
\end{remark}
\begin{remark} \label{rem:largenegativesquare}
In the proof of Proposition~\ref{prop:primitivepoincare}, the assumption that the vectors are of large negative square is crucial, for the following reason: Any integral class of negative square in $\Hlprim(X, \mathbb Z)$ can be made Hodge on a Noether--Lefschetz divisor of the period domain of $L$-polarized K3 surfaces or of cubic fourfolds, but one must avoid certain classes of small negative square to ensure both that the point is in the image of the period map as well as the existence of a Poincar\'e sheaf $\cP$ on $X \times X$, which induces the Hodge isometry
\[\Phi^{\widetilde{\rH}}_{\cP} \colon \Hlprim^2(X, \mathbb Q) \iso \Hlprim^2(X, \mathbb Q).\]
Indeed, in order to obtain a Poincar\'e sheaf on the Beauville-Mukai system $\overline{\Pic}^0(S/|L|)$ for a quasi-polarized K3 surface $(S, L)$, one needs the line bundle $L$ to be ample and all members of the linear system $|L|$ to be integral. Since both conditions are open in the moduli space of quasi-polarized K3 surfaces, this can be ensured by avoiding a finite set of Noether--Lefschetz divisors in the moduli space of $L$-quasi-polarized K3 surfaces (and further closed subsets of codimension at least two).
To obtain a Poincar\'e sheaf on the LSV construction associated to a cubic fourfold, we must instead avoid the Hassett divisors $\cC_2$ and $\cC_6$ to ensure that the Hodge structure is in the image of the period map for cubic fourfolds, and $\cC_8$ and $\cC_{12}$ (as well as a codimension $2$ subset consisting of cubic fourfolds with highly singular hyperplane sections) in order for the cubic fourfold to be very good, see \cite[Corollary 1.4]{marquandviktorova}.
\end{remark}

\begin{remark}
More generally, the strategy of the proof of Proposition~\ref{prop:primitivepoincare} can be used to show that if $U \subset \cD$ is an open subset in a positive-dimensional period domain parametrizing Hodge structures of K3-type with signature $(1,r)$, where the very general point has no integral class in $\rH^{1,1}$, then any rational Hodge isometry $\phi$ that exists over $U$ satisfies $\phi = \pm \mathrm{id}$.
\end{remark}

\section{The main construction}
\label{sec:main_construction}
As an application of the theory of MBM classes, we begin by constructing examples of projective hyperk\"ahler manifolds which admit transverse Lagrangian fibrations. 
\begin{construction} \label{construction:hklattice}
    Fix some integer $k > 0$, and consider the rank two lattice $U(k) = \langle \isocls_1, \isocls_2 \rangle$ with intersection form 
    \[\begin{pmatrix}
        0 & k \\
        k & 0
    \end{pmatrix}.\]
    Fix a copy of $U^{\oplus 2}$ with basis $\gamma_1,\mu_1,\gamma_2,\mu_2$ such that $(\gamma_i,\gamma_j) = (\mu_i,\mu_j) = 0$ and $(\gamma_i,\mu_j) = \delta_{ij}.$ We fix once and for all the primitive embedding 
    \begin{equation}
        \iota \colon U(k) \hookrightarrow U^{\oplus 2}
    \end{equation}
    given by 
    \begin{align*}
        \isocls_1 &\mapsto \gamma_1 \\
        \isocls_2 &\mapsto k\mu_1+\gamma_2.
    \end{align*}
    In particular, the classes $\isocls_1,\isocls_2$ are of divisibility one in $U^{\oplus 2}$. Recall that 
    \[\Lambda_{\Kthreen} \coloneq U^{\oplus 3} \oplus E_8(-1)^{\oplus 2} \oplus \bbZ(2-2n)\]
    and 
    \[\Lambda_{\OG} \coloneq U^{\oplus 3} \oplus E_8(-1)^{\oplus 2} \oplus A_2(-1).\]
    As a consequence, there are primitive embeddings $\iota_{\Kthreen}$ and $\iota_{\OG}$ given by the embedding $\iota$ into the first two copies of $U$ and trivial map in all other factors. Since the signature of the orthogonal complements $\iota_{\Kthreen}(U(k))^\perp$ and $\iota_{\OG}(U(k))^\perp$ are $(2,r)$ for some $r > 0$, the moduli spaces of hyperkähler manifolds of $\Kthreen$- and $\OG$-type polarized by the lattice $U(k)$ are non-empty by surjectivity of the period map. In particular, there is a hyperkähler manifold $Y$ in either deformation type satisfying $\Pic(Y) = U(k) = \langle \isocls_1, \isocls_2 \rangle \subset H^2(Y, \mathbb Z) \simeq \Lambda_{\Kthreen/\OG}$.
    Moreover, up to replacing $\isocls_1, \isocls_2$ with $-\isocls_1, -\isocls_2$ we may assume that $\isocls_1,\isocls_2$ lie on the boundary of $\mathrm{Pos}(Y),$ so in particular $a\isocls_1 + b\isocls_2 \in \mathrm{Pos}(Y) = \mathrm{Kah}(Y)$ for $a, b > 0$, and $Y$ is projective. \par 
    In the $\OG$ case, observe that there is a surjection 
    \[\{[v] \in \bbP((\Lambda_{\OG})_\bbC) \mid v \in \isocls_i^\perp, v^2 = 0, (v,\overline v) > 0\} \twoheadrightarrow \{[v] \in \bbP(\isocls_i^\perp/\isocls_i) \mid v^2 = 0, (v,\overline{v})> 0\},\]
    so in particular for a very general choice of period we may assume that $\isocls_1^\perp/\isocls_1$ (resp. $\isocls_2^\perp/\isocls_2$) is Hodge-isometric to $\Hprim^4(-,\bbZ)(1)$ for a very good cubic fourfold. 
\end{construction}
\begin{example}
\cite[Sec.\ 7]{kapustkaconstruction}
    Let $S$ be a polarized K3 surface with $\Pic(S) = \mathbb Z  h$. Pick $B' \in H^2(S, \mathbb Z)$ primitive with $B' \cdot h = 0$ and $(B')^2 = 0$. Set $B \coloneqq \frac{1}{k}B$. Then one has
    $$\Pic(M_{(0,h, 0)}(S, B)) = \begin{pmatrix}0 & k \\ k & 0\end{pmatrix}.$$
    Then $Y = M_{(0, h, 0)}(S, B)$ is an explicit moduli space of sheaves realizing a hyperk\"ahler of $\Kthreen$-type of the desired form.
\end{example}

As a consequence of the particular choice of intersection form, the birational geometry of these examples is highly constrained. 
\begin{proposition}\label{prop:transverse}
    Let $Y$ be a hyperk\"ahler manifold of $\Kthreen$- or $\OG$-type arising from Construction~\ref{construction:hklattice}. If $k > \mbmconst$, where $\mbmconst$ is the bound on the square of MBM classes from Theorem~\ref{thm:mbmbound}, then any birational map $\phi \colon Y \dashrightarrow Y'$ of hyperk\"ahler manifolds is in fact biregular. \par 
    Moreover, $Y$ admits two Lagrangian fibrations with transverse fibers, i.e., if $$\pi_1: Y \to \bbP^n \text{ and } \pi_2 \colon Y \to \mathbb P^n$$ are the fibrations then $\pi_1^{-1}(p) \cap \pi_2^{-1}(q)$ is finite for all $p,q \in \bbP^n.$
\end{proposition}
\begin{proof}
    We first observe that $Y$ necessarily has no MBM classes, since for any integral algebraic vector $z \in \Pic(X),$ by construction we know that $k \mid v^2$, so if $v^2 < 0$ it follows that in fact $v^2 \leq -k < -\mbmconst$. As a consequence we deduce by Theorem \ref{thm:mbmcone} that $\mathrm{Kah}(Y)=\mathrm{Pos}(Y).$ \par
    Suppose $\phi \colon Y \dashrightarrow Y'$ is a birational map of hyperk\"ahler manifolds; in particular, it must be an isomorphism in codimension 1, see \cite[Sec.\ 2.2]{birationalhkcodim2}. To observe that $\phi$ is biregular, it suffices to observe that $\phi^* \colon H^2(Y', \mathbb Z) \iso H^2(Y, \mathbb Z)$ sends a K\"ahler class to a K\"ahler class by \cite[Corollary 3.3]{fujiki}. But since $\phi^* \colon \rH^2(Y',\bbZ) \iso \rH^2(Y,\bbZ)$ is an isometry with respect to the BBF form and thus $\phi^*\mathrm{Pos}(Y') = \mathrm{Pos}(Y)$, it follows in particular that every K\"ahler class on $Y'$ pulls back to a class in $\mathrm{Pos}(Y) = \mathrm{Kah}(Y)$, so in fact $\psi$ is biregular. \par 
    Now since $\mathrm{Kah}(Y) = \mathrm{Pos}(Y)$, it follows that the classes $\isocls_1, \isocls_2 \in \rH^2(Y,\bbZ)$ are on the boundary of the K\"ahler cone, and are in particular nef, so by the SYZ conjecture for $\Kthreen$ or $\OG$ manifolds \cite{matsushita,mongardionorati}, there exist Lagrangian fibrations $\pi_i \colon Y \to \bbP^n$ such that $$\isocls_i = c_1(\pi_i^*\cO_{\bbP^n}(1)).$$
    Consider the induced morphism $\pi_1\times\pi_2 :Y \to \bbP^n \times \bbP^n$; the fibers of this morphism are precisely the intersections of the fibers of $\pi_1$ and $\pi_2$. By construction, the class $$\isocls_1 + \isocls_2 = c_1((\pi_1\times\pi_2)^*\cO_{\bbP^n\times\bbP^n}(1,1))$$ is ample, so it follows that the morphism $\pi_1 \times \pi_2$ cannot contract any curves and must be finite. 
\end{proof}

\begin{remark} \label{rmk:intersectionrmk}
    By using the BBF form, we can in fact calculate how many points the transverse fibers intersect in. Observe that the classes of the fibers along the two Lagrangian fibrations are given by $\isocls_1^n$ and $\isocls_2^n$. By the definition of the BBF form, we have equalities 
    \[ \frac{(2n)!}{n!n!}\int_Y\isocls_1^n\isocls_2^n = \int_Y(\isocls_1+\isocls_2)^{2n} = c_Y\frac{(2n)!}{n!2^n}((\isocls_1+\isocls_2)^2)^n = c_Y\frac{(2n)!}{n!2^n}(2k)^n,\]
    and since $c_Y = 1$ for both deformation types we find that 
    \begin{equation} \label{eqn:intersectioneqn}
        \int_Y\isocls_1^n\isocls_2^n = n!k^n. 
    \end{equation}
    It follows that, up to multiplicity, all fibers of $\pi_1$ intersect with all fibers of $\pi_2$ in $n!k^n$ points. 
\end{remark}
Owing to Theorem \ref{thm:lagfibthm} in the OG10 case or \cite[Theorem 1.5]{markmanlagrangian} in the $\Kthreen$ case\footnote{Note that the non-speciality condition was removed in \cite{soldatenkovverbitsky,abasheva}.}, both Lagrangian fibrations arising in Proposition \ref{prop:transverse} are therefore realized by Tate--Shafarevich twists of compactified abelian fibrations (either of the form $\overline\Pic^0$ or $\overline\IJ$). We summarize the resulting picture with the following diagram:
\begin{equation}
    \begin{tikzcd}
        X \arrow[rd, "\pi_X"'] & & Y \arrow[ld, "\pi_1"] \arrow[rd, "\pi_2"'] & & Z \arrow[ld, "\pi_Z"] \\
        & \bbP^n & & \bbP^n & 
    \end{tikzcd}
\end{equation}
where either $X = \overline\Pic^0(S_X, L_X)$ and $Z = \overline\Pic^0(S_Z, L_Z)$ for polarized K3 surfaces $(S_X, L_X)$ and $ (S_Y, L_Y)$, or $X = \overline\IJ(C_X)$ and $Z = \overline\IJ(C_Z)$ for very good cubic fourfolds $C_X$ and $C_Z$.

In the $\Kthreen$ case, the associated K3 surfaces $S_X$ and $S_Y$ have Picard rank $1$ and the line bundles $L_X, L_Y$ are generators of $\Pic(S_X)$ and $\Pic(S_Y)$ since we have $\rho(X) = 2$ and the divisibility of $\pi_X^*\mathcal O_{\mathbb P^n}(1)$ is one by construction, cf.\ \cite[Theorem 1.5]{markmanlagrangian}.
\subsection{The convolution}
Suppose we are in the situation of Construction \ref{construction:hklattice}, and assume that we have taken $k$ large enough that Proposition \ref{prop:transverse} holds. By Theorem \ref{thm:poincare}, there exist twisted Poincar\'e sheaves \[\cP_{X,Y} \in \Db(X\times Y,\mathrm{pr}_1^*\alpha_X) \quad\text{ and }\quad \cP_{Y,Z} \in \Db(Y \times Z, \mathrm{pr}_2^*\alpha_Z)\] inducing Fourier--Mukai equivalences. Let $\cQ \in \Db(X \times Z, \alpha_X \boxtimes \alpha_Z)$ denote the Fourier--Mukai kernel of the composition
\[\Phi_{\cP_{Y, Z}} \circ \Phi_{\cP_{X, Y}} \colon \Db(X, -\alpha_X) \to \Db(Y) \to \Db(Z, \alpha_Z).\]
\begin{lemma}\label{lem:kernellocfree}
The twisted sheaf $\cQ$ is a twisted vector bundle of rank $n!k^n$. 
\end{lemma}

\begin{proof}[Proof of Lemma~\ref{lem:kernellocfree}]
Fix a point $x \in X$. Note that $\Phi_{\cP_{X, Y}}(\cO_x)$ is a sheaf of rank one on the fiber over the image of $x$ in $\mathbb P^n$. By construction, the restriction of $\pi_2$ to this fiber is finite of degree $n!k^n$ by Proposition \ref{prop:transverse} and Remark \ref{rmk:intersectionrmk}. Thus, the claim follows from Proposition \ref{prop:imageoffinlag} if we can show that $\Phi_{\mathcal P_{X, Y}}(\cO_x)$ is Cohen--Macaulay on $Y$. As the Poincar\'e sheaf is flat over both factors, the result follows from \cite[Lemma 2.10]{yulsv}.
\end{proof}

\begin{lemma} \label{lem:isotropicpreserved}
    Let $\psi \coloneq \widetilde\Psi_\cQ|_{\rH^2(X,\bbQ)} \colon \rH^2(X,\bbQ) \to \rH^2(Z,\bbQ)$ as defined in Theorem \ref{thm:hyperholdef}. Then 
    \[\psi(\isocls_X) = \cnstsign \isocls_Z, \qquad\qquad \psi({\theta_X}) = \cnstsign\theta_Z.\]
    for some fixed constant $\cnstsign = \pm 1$. Here $\isocls_X = c_1(\pi_X^*\cO_{\bbP^n}(1)), \isocls_Z = c_1(\pi_Z^*\cO_{\bbP^n}(1))$, and $\theta_X,\theta_Z$ are isotropic generalized theta divisors for $\pi_X$ and $\pi_Z$, as in Section \ref{sec:cohomFM}.
\end{lemma}
\begin{proof}
    Combining the definition of $\widetilde\Psi_\cQ$ with Proposition \ref{prop:extmukaipoincare} (along with Remark \ref{rmk:transposermk}), we observe that this isometry can be factored as 
    \[
    \begin{tikzcd}
        \rHt(X,\bbQ) \arrow[d, "\cdot \exp(B_X')"] \arrow[rrrr, "\widetilde\Psi_\cQ"]& & & & \rHt(Z,\bbQ) \\
        \rHt(X,-B_X,\bbQ) \arrow[d, equals] & & & & \rHt(Z,B_Z,\bbQ) \arrow[u, "\cdot \exp(B_Z')"] \\
        \rHt(X,\bbQ) \arrow[r, "\Phi^{\rHt}_{\cP_X}"] & \rHt(X,\bbQ) \arrow[r, "\rho_{X,Y}"] & \rHt(Y,\bbQ) \arrow[r, "\rho_{Y,Z}"] & \rHt(Z,\bbQ) \arrow[r, "\Phi^{\rHt}_{\cP_Z}"] & \rHt(Z,\bbQ) \arrow[u, equals]
    \end{tikzcd}
    \]
    where:
    \begin{itemize}
        \item $B_X$ and $B_Z$ are the B-fields for $\alpha_X$ and $\alpha_Z$ guaranteed by Proposition \ref{prop:extmukaipoincare} respectively,
        \item $\mathrm{pr}_1^*B_X' + \mathrm{pr}_2^*B_Z' = -\frac{c_1^B(\cQ)}{\rank(\cQ)},$ making $B_X'$ and $B_Z'$ respectively B-fields for $\alpha_X$ and $\alpha_Z$ which differ from $B_X$ and $B_Z$ by classes in $\rH^{1,1}_\bbQ$, 
        \item $\rho_{X,Y}$ and $\rho_{Y,Z}$ are respectively the parallel transports from $X$ to $Y$ and from $Y$ to $Z$ along degenerate twistor lines, and 
        \item $\cP_X \in \Db(X \times X)$ and $\cP_Z \in \Db(Z \times Z)$ are the untwisted Poincar\'e sheaves. 
    \end{itemize}
    Note that while the five maps $\widetilde{\Psi}_\cQ$, $\exp(B_X')\cdot (-)$, $\exp(B_Z')\cdot(-),\, \Phi_{\cP_X}^{\rHt}$ and $\Phi_{\cP_Z}^{\rHt}$ are Hodge isometries, the equalities and parallel transports are not. \par 
    Using the explicit descriptions from Example \ref{ex:explambdallv} and Proposition \ref{prop:rank4sublattice}, along with the fact that the parallel transport sends $f \mapsto f$ and preserves the associated isotropic classes, i.e.,
    \[\rho_{X,Y}(\isocls_X) = \isocls_1 \text{ and }\rho_{Y,Z}(\isocls_2) = \isocls_Z,\] we see that the image of $\isocls_X$ in $\rHt(Y,\bbQ)$ is 
    \begin{align*}
        \rho_{X,Y}\circ \Phi_{\cP_X}^{\rHt}(\exp(B_X')\cdot\isocls_X) &= \rho_{X,Y} \circ \Phi_{\cP_X}^{\rHt}(\isocls_X+(B_X',\isocls_X)f) \\
        &= \rho_{X,Y}(-\cnstsign_Xf+(B'_X,\isocls_X)\cnstsign_X\isocls_X) \\
        &= -\cnstsign_Xf + \cnstsign_X(B'_X,\isocls_X)\isocls_1
    \end{align*}
    for some $\cnstsign_X = \pm 1$. On the other hand the image of $f$ in $\rHt(Y,\bbQ)$ is 
    \[\rho_{X,Y} \circ \Phi_{\cP_X}^{\rHt}(\exp(B_X')\cdot f) = \cnstsign_X\isocls_1.\]
    However, the composition $\widetilde\Psi_\cQ \colon \rHt(X,\bbQ) \to \rHt(Z,\bbQ)$ is degree-reversing by Lemma \ref{lem:posrankdegreversing}, so in particular the image of $f$ in $\rHt(Z,\bbQ)$ is $K\cdot e$ for some rational constant $K \neq 0$. It follows that the image of $\cnstsign_X\isocls_1$ under the composition $\rHt(Y,\bbQ) \to \rHt(Z,\bbQ)$ is $K\cdot e.$ A similar calculation to before also shows
    \[\exp(B'_Z)\cdot(\Phi_{\cP_Z}^{\rHt} \circ \rho_{Y,Z}(f)) = \cnstsign_Z\isocls_Z + \cnstsign_Z(\isocls_Z,B_Z')f\]
    for some $\cnstsign_Z = \pm 1$, so we see 
    \[\widetilde\Psi_\cQ(\isocls_X) = (B'_X,\isocls_X)K\cdot e -\cnstsign_X\cnstsign_Z\isocls_Z - \cnstsign_X\cnstsign_Z(\isocls_Z,B_Z')f.\]
    Since $\widetilde\Psi_\cQ$ preserves degree 2 cohomology, it follows that $(B'_X,\isocls_X) = (\isocls_Z,B_Z') = 0$ and 
    \[\widetilde\Psi_\cQ(\isocls_X) = -\cnstsign_X\cnstsign_Z\isocls_Z.\]
    Since we have rational Hodge isometries $\rHt(X,\bbQ) \simeq \rHt(Y,\bbQ) \simeq \rHt(Z,\bbQ)$, the Picard ranks of $X,Y$ and $Z$ must coincide, and by the choices made in Construction \ref{construction:hklattice}, it follows that they are all of Picard rank 2. In particular, since $\widetilde\Psi_\cQ$ is isometric and preserves $\rH^2(-,\bbQ)$, we have
    \[\widetilde\Psi_\cQ(\theta_Z) = -\cnstsign_X\cnstsign_Z\theta_Z,\]
    which allows us to conclude.
\end{proof}

The strategy to show that $\cQ$ is hyperholomorphic relies on knowing the stability of the fibers of $\cQ$ along the two projections with respect to appropriate polarizations, using Lemma \ref{lem:stabilitylemma}. The following easy lemma allows us to extend that stability to a nearby classes in $\mathrm{Kah}(Y)$ which are not necessarily algebraic.
\begin{lemma}\label{lem:amptokahler}
    For any hyperk\"ahler manifold $Y$, any twisted vector bundle $\cE$ on $Y$ and any polarization $H \in \mathrm{Amp}(Y)_\bbR$ such that $\cE$ is $\mu_H$-stable, $\cE$ is also $\mu_\omega$-stable for any K\"ahler class of the form $\omega = H + \tau$ where $\tau \in \Htr^2(Y,\bbR)$ is a transcendental cohomology class, i.e., $\tau$ is BBF-orthogonal to $\rH^{1,1}(Y,\bbZ).$
\end{lemma}
\begin{proof}
    By \cite[Lemma 3.7]{ogradymodular} (which as observed in \cite[Proposition 6.5]{bottinitenfolds} carries over to the twisted case, and which easily generalizes from polarizations to K\"ahler classes), the twisted vector bundle $\cE$ is $\mu_H$-stable (resp. $\mu_\omega$-stable) if and only if $(\lambda_{\cF,\cE},H) < 0$ (resp. $(\lambda_{\cF,\cE},\omega) < 0$) where 
    \[\lambda_{\cF,\cE} \coloneq c_1(\cE^\vee \otimes \cF)\]
    for any non-zero twisted subsheaf $\cF \subset \cE$. Since $\cE^\vee \otimes \cF$ is an ordinary untwisted sheaf, $c_1(\cE^\vee \otimes \cF)$ is a rational Hodge class. Since the difference $\mu - H$ is transcendental, it follows that $$(\lambda_{\cF,\cE},H) = (\lambda_{\cF,\cE},\mu)$$ so the result follows. 
\end{proof}

\begin{theorem} \label{thm:hyperholexistence}
    There exists some $\cnstsign = \pm 1$ such that the twisted vector bundle $\cQ$ satisfies the conditions of Theorem \ref{thm:hyperholdef}; i.e., it is a hyperholomorphic twisted vector bundle on $X \times Z$. 
\end{theorem}
\begin{proof}
    We first observe that for any $x \in X$, the twisted vector bundle $$\cQ|_{\{x\}\times Z} \simeq \Phi_{\cQ}(\cO_x) \simeq \Phi_{\cP_{Y,Z}}(\Phi_{\cP_{X,Y}}(\cO_x))$$ is $\mu_{H_Z}$-stable with respect to any $a(\cQ|_{\{x\}\times Z})$-suitable polarization $H_Z$ on $Z$. It follows from the definition of $a(\cE)$ for twisted modular vector bundles $\cE$ \cite[Definition 6.4]{beckmann} (which itself relies on \cite[Definitions 1.1 and 3.3]{ogradymodular}) that this quantity is deformation-invariant, hence independent of the choice of point $x \in X$. Hence write $a_Z \coloneq a(\cQ|_{\{x\}\times Z})$ for some (equivalently, any) $x \in X$. A symmetric argument shows that $\cQ|_{X\times\{z\}}$ is $\mu_{H_X}$-stable $H_X$ an $a_X$-suitable polarization, where $a_X \coloneq a(\cQ|_{X\times\{z\}})$ for $z \in Z$. \par 
    This is enough to deduce that $\cQ$ is $\mu_H$-stable where $H = \mathrm{pr}_1^*H_X + \mathrm{pr}_2^*H_Z$ for $H_X$ and $H_Z$ respectively $a_X$- and $a_Z$-suitable polarizations. It is enough to show that for some constant $\cnstsign = \pm 1$, there is an open subcone $\cC_X \subset \mathrm{Amp}(X)_\bbR$ consisting of $a_X$-suitable classes which are sent under the rational Hodge isometry $\cnstsign\cdot \psi = \cnstsign\cdot \widetilde\Psi_\cQ|_{\rH^2(X,\bbQ)} \colon \rH^2(X,\bbQ) \to \rH^2(Z,\bbQ)$ to $a_Z$-suitable classes. Indeed, applying Lemma \ref{lem:amptokahler} will allow us to enlarge this to an open subcone $\cC_X \subset \mathrm{Kah}(X)_\bbR$ such that the hypotheses of Theorem \ref{thm:hyperholdef} are satisfied. \par
    By Lemma \ref{lem:isotropicpreserved}, there is a constant $\cnstsign = \pm 1$ so that $\cnstsign\psi(\isocls_X) = \isocls_Z$ and $\cnstsign\psi(\theta_X) = \theta_Z$. Since $\isocls_X$ and $\isocls_Z$ are nef and are therefore on the boundary of the ample cone, by moving from $\isocls_X$ in the direction of the interior of the ample cone we may find a polarization $H$ on $X$ such that $\cnstsign\psi(H)$ is ample on $Z$. By the openness of ampleness we may enlarge this to a neighborhood $H \in U \subset \mathrm{Amp}(X)$ such that $\cnstsign\psi(U) \subset \mathrm{Amp}(Z).$ \par 
    Let $d = \max\{0,a_X\cdot\sup_{H_X \in U}(H_X,\isocls_X),a_Z\cdot\sup_{H_Z \in \cnstsign\psi(U)}(H_Z,\isocls_Z)\}.$ Then by a bound of Friedman \cite[Corollary 5.14]{ogradymodular2}, for any $H_X \in U$ the polarization $H_X + d\isocls_X/2$ is $a_X$-suitable and $C\psi(H_X) + d\isocls_Z/2$ is $a_Z$-suitable. Taking the cone spanned by polarizations in $U + d\isocls_X/2$ yields the necessary subcone $\cC_X \subset \mathrm{Amp}(X)_\bbR$.
\end{proof}
We then have an easy corollary of the deformation of hyperholomorphic sheaves over twistor lines.
\begin{corollary} \label{cor:hyperholequivfromdef}
    Let $X'$ be any hyperk\"ahler manifold of $\Kthreen$- or $\OG$-type. Then there exists another hyperk\"ahler $Z'$ of the same deformation type and a twisted vector bundle $\cQ'$ on $X' \times Z'$ realizing a twisted derived equivalence 
    \[\Phi_{\cQ'} \colon \Db(X',-\alpha_{X'}) \iso \Db(Z',\alpha_{Z'}).\]
\end{corollary}
\begin{proof}
    Consider the hyperholomorphic sheaf $\cQ$ on $X \times Z$ from Theorem \ref{thm:hyperholexistence}, and let $$\phi = \cnstsign\psi \colon \rH^{2}(X,\bbQ) \iso \rH^2(Y,\bbQ)$$ be the rational Hodge isometry which sends a K\"ahler class to a K\"ahler class. Fix markings $\marking_X$ of $X$ and $\marking_Z$ of $Z$. \par 
    Since $X$ and $X'$ are deformation equivalent, there exists a marking $\marking_{X'}$ of $X'$ so that $(X,\marking_X)$ and $(X',\marking_{X'})$ are in the same connected component of $\fM_\Lambda$ where $\Lambda = \Lambda_{\Kthreen}$ or $\Lambda = \Lambda_{\OG}$ depending on the deformation type of $X'$. By Lemma \cite[Lemma 5.14]{markmanalgebraic} there exists a marked hyperk\"ahler $(Z',\marking_{Z'})$ such that 
    \[(X,\marking_X,Z,\marking_Z) \text{ and } (X',\marking_{X'},Z',\marking_{Z'})\]
    lie in the same connected component of $\fM_{\phi}$. From Theorems \ref{thm:hyperholdef} and \ref{thm:hyperholequivalence} it follows that there exists a vector bundle $\cQ'$ inducing a twisted derived equivalence of $X'$ and $Z'$. 
\end{proof}
We also know from Corollary \ref{cor:bfielddeformationcor} that the Brauer classes $\alpha_{X'} \in \Br(X')$ and $\alpha_{Z'} \in \Br(Z')$ appearing in the statement of the theorem above are determined by parallel transport of the B-fields $B_X'$ and $B_Z'$. 
\section{Applications to \texorpdfstring{hyperk\"ahler}{hyperkähler} manifolds of OG10-type}
\label{sec:applications}
As a consequence of the existence of these vector bundles on $\OG$ manifolds, we are able to translate a number of results known in the $\Kthreen$ case to the $\OG$ case. 
\subsection{The Lefschetz standard conjecture for hyperk\texorpdfstring{\"a}{ä}hler manifolds of OG10-type}
We begin with an easy consequence of the existence of hyperholomorphic derived equivalences which follows from a result of Markman. 
\begin{theorem} \label{thm:lsc}
    Suppose $X$ is a projective hyperk\"ahler manifold of OG10-type. Then the Lefschetz standard conjecture holds for $X$. 
\end{theorem}
\begin{proof}
    This follows immediately from using the degree-reversing Hodge isometry coming from the Fourier--Mukai kernel of nonzero rank of Corollary \ref{cor:hyperholequivfromdef} in the OG10 case and \cite[Lemma 1.6]{markmanalgebraic}.
\end{proof}
\subsection{\texorpdfstring{$k$}{k}-cyclic Hodge isometries for the primitive cohomology of cubic fourfolds}
We return to the setting of Construction \ref{construction:hklattice}, assuming now that the hyperk\"ahler manifolds are all of OG10-type; as before, we assume that $k$ is large enough so that Proposition \ref{prop:transverse} holds. In particular, there are very good cubic fourfolds $C_X$ and $C_Z$ and a diagram of the form 
\begin{equation}
    \begin{tikzcd}
        X = \overline\IJ(C_X) \arrow[rd, "\pi_X"'] & & \overline\IJ_s(C_X) = Y = \overline\IJ_t(C_Z) \arrow[ld, "\pi_1"] \arrow[rd, "\pi_2"'] & & \overline\IJ(C_Z) = Z, \arrow[ld, "\pi_Z"] \\
        & \bbP^n & & \bbP^n & 
    \end{tikzcd}
\end{equation}
where $\overline\IJ_s(C_X)$ is a Tate--Shafarevich twist of $X = \overline\IJ(C_X)$ and $\overline\IJ_t(C_Z)$ is a Tate--Shafarevich twist of $Z = \overline\IJ(C_Z).$ As before, we take 
\[\cP_{X,Y} \in \Db(X \times Y,\mathrm{pr}_1^*\alpha_X), \quad \cP_{Y,Z} \in \Db(Y \times Z,\mathrm{pr}_2^*\alpha_Z)\]
to be the twisted Poincar\'e sheaves and 
\[\cQ \in \Db(X\times Z,\alpha_X \boxtimes \alpha_Z)\]
their convolution, along with natural B-field lifts $B_X$ of $\alpha_X$ and $B_Z$ of $\alpha_Z$ as ensured by Proposition \ref{prop:extmukaipoincare} and potentially different B-field lifts $B_X'$ of $\alpha_X$ and $B_Z'$ of $\alpha_Z$ given by
\[\mathrm{pr}_1^*B_X' + \mathrm{pr}_2^*B_Z' = -\frac{c_1^{{\mathrm{pr}_1^*B_X+\mathrm{pr}_2^*B_Z}}(\cQ)}{\rank(\cQ)},\]
see equation \eqref{eqn:newBfields}.
\par 
By construction, $X$, $Y$ and $Z$ are endowed with explicit markings. In particular, we identify 
\[\rH^2(Y,\bbZ) \simeq  \langle \gamma_1,\mu_1\rangle \oplus \langle \gamma_2,\mu_2\rangle \oplus U \oplus E_8(-1)^{\oplus 2}\oplus A_2(-1),\]
where $\isocls_1 \mapsto \gamma_1$ and $\isocls_2 \mapsto k\mu_1 + \gamma_2$ form a basis for $\Pic(Y) \simeq U(k)$, and $\gamma_1,\mu_1$ (resp. $\gamma_2,\mu_2$) are the natural basis for the first (resp. second) copy of the hyperbolic plane $U$ in $\Lambda_{\OG}$. In particular we find induced isomorphisms 
\[\isocls_1^\perp/\isocls_1 \simeq \langle \gamma_2,\mu_2\rangle \oplus U \oplus E_8(-1)^{\oplus 2} \oplus A_2(-1)\]
and 
\[\isocls_2^\perp/\isocls_2 \simeq \langle \mu_1,\gamma_1-k\mu_2\rangle \oplus U \oplus E_8(-1)^{\oplus 2} \oplus A_2(-1).\]
\par 
Using \cite[Theorem 4.6]{ijtwists}, we identify 
\[\rH^2(X,\bbZ) \simeq  \langle \isocls_X,\theta_X\rangle \oplus \isocls_1^\perp/\isocls_1 \simeq \langle \isocls_X ,\theta_X\rangle \oplus \langle \gamma_2,\mu_2\rangle\oplus U \oplus E_8(-1)^{\oplus 2} \oplus A_2(-1),\]
where $\isocls_X,\theta_X$ again form the basis of a hyperbolic plane, and with the parallel transport along the degenerate twistor line 
\[\rH^2(X,\bbZ) \iso \rH^2(Y,\bbZ)\]
sending $\isocls_X \mapsto \isocls_1 = \gamma_1$, $\theta_X \mapsto \mu_1$ and otherwise mapping by the obvious inclusion. \par 
Similarly, we identify 
\[\rH^2(Z,\bbZ) \simeq \langle \isocls_Z,\theta_Z\rangle \oplus\isocls_2^\perp/\isocls_2 \simeq \langle \isocls_Z,\theta_Z\rangle\oplus \langle \mu_1,\gamma_1-k\mu_2\rangle \oplus U \oplus E_8(-1)^{\oplus 2} \oplus A_2(-1)\]
with the parallel transport 
\[\rH^2(Z,\bbZ) \iso \rH^2(Y,\bbZ)\]
sending $\isocls_Z \mapsto \isocls_2 = k\mu_1 + \gamma_2$, $\theta_Z \mapsto \mu_2$, and mapping via the inclusion on the rest. \par 
As a consequence of this explicit lattice-theoretic description of the parallel transports, we can determine explicitly the B-fields $B_X'$ and $B_Z'$. For the following results, let $\epsilon = \pm1$ denote the sign of the restriction of the Poincar\'e isometries $\Phi_{\cP_X}^{\rHt}|_{\Hlprim^2(X,\bbZ)}$ and $\Phi_{\cP_Z}^{\rHt}|_{\Hlprim^2(Z,\bbZ)}$ ensured by Proposition \ref{prop:primitivepoincare}.

\begin{lemma}\label{lem:explicitBfields}
    With the choice of markings above, we have an identification 
    \[B_X' = 3\isocls_X - \epsilon\frac{\gamma_2}{k}\]
    and 
    \[B_Z' = 3\isocls_Z - \epsilon\frac{\gamma_1-k\mu_2}{k},\]
    where $\gamma_2 \in\isocls_1^\perp/\isocls_1 \simeq \Hlprim^2(X,\bbZ)$ and $\gamma_1-k\mu_2 \in \isocls_2^\perp/\isocls_2 \simeq \Hlprim^2(Z,\bbZ)$ are isotropic vectors of divisibility one. 
\end{lemma}
\begin{proof}
    It is enough to consider the second identity, as the argument for the first is analogous. By definition, $B_Z' = -c_1^{B_Z}(\cQ|_{\{x\}\times Z})/\rank(\cQ)$. We determine this (twisted) Chern character by understanding the extended Mukai vector of $\cQ|_{\{x\}\times Z}$. \par 
    By \cite[Definition 4.11]{beckmann}, it suffices to understand the image of $f \in \rHt(X,-B_X,\bbQ)$ under $\Phi_{\cQ}^{\rHt}$, which by Proposition \ref{prop:extmukaipoincare} coincides with the composition 
    \[
    \begin{tikzcd}
        \rHt(X,-B_X,\bbQ) \arrow[d, equals] \arrow[rrrr,"\Phi_{\cQ}^{\rHt}"] & & & & \rHt(Z,B_Z,\bbQ)  \\
        \rHt(X,\bbQ) \arrow[r, "\Phi^{\rHt}_{\cP_X}"] & \rHt(X,\bbQ) \arrow[r, "\rho_{X,Y}"] & \rHt(Y,\bbQ) \arrow[r, "\rho_{Y,Z}"] & \rHt(Z,\bbQ) \arrow[r, "\Phi^{\rHt}_{\cP_Z}"] & \rHt(Z,\bbQ) \arrow[u, equals]
    \end{tikzcd}
    \]
    It is a straightforward calculation in terms of the parallel transports and Propositions \ref{prop:rank4sublattice} and \ref{prop:primitivepoincare} that 
    \begin{align*}
        \Phi^{\rHt}_\cQ(f) &= \Phi^{\rHt}_{\cP_Z} \circ \rho_{Y,Z} \circ \rho_{X,Y} \circ \Phi_{\cP_X}^{\rHt}(f) \\
        &= \Phi^{\rHt}_{\cP_Z} \circ \rho_{Y,Z} (\isocls_1) \\
        &= \Phi^{\rHt}_{\cP_Z} \circ \rho_{Y,Z} (k\mu_2 + \gamma_1-k\mu_2) \\
        &= \Phi^{\rHt}_{\cP_Z} (k\theta_Z + (\gamma_1 - k\mu_2)) \\
        &= ke - 3k\isocls_Z  + \epsilon(\gamma_1-k\mu_2),
    \end{align*}
    where $\gamma_1 - k\mu_2 \in \isocls_2^\perp/\isocls_2$ is an isotropic vector with divisibility one. Then by \cite[Lemma 4.13(iv)]{beckmann}, one can calculate 
    \[\rank(\cQ) = n!k^n, \qquad c_1^{B_Z}(\cQ|_{\{x\}\times Z}) = n!k^{n-1}\cdot(-3k\isocls_Z \pm (\gamma_1-k\mu_2))),\]
    and the description of $B_Z'$ follows. 
\end{proof}
\begin{remark}
    It is worth highlighting one potential point of confusion: since the other B-field $B_X'$ coincides with $-c_1^{B_X}(\cE|_{X\times\{z\}})/\rank(\cE)$, to calculate it as above, one should really use the functor 
    \[\Phi_{\cQ} \colon \Db(Z,-\alpha_Z) \to \Db(X,\alpha_X)\] 
    induced by $\cQ$ viewed as a functor in the opposite direction, rather than $\Phi_{\cQ}^{-1}$. 
\end{remark}
Note that there are canonical injections \[i_1 \colon \langle \isocls_1,\isocls_2\rangle^\perp \hookrightarrow \isocls_1^\perp/\isocls_1 \text{ and } i_2 \colon \langle \isocls_1,\isocls_2\rangle^\perp \hookrightarrow \isocls_2^\perp/\isocls_2.\]
It is easy to check that both injections are of index $k$ from the explicit description
\[\langle \isocls_1,\isocls_2\rangle^\perp = \langle \gamma_2,\gamma_1 - k\mu_2 \rangle \oplus U \oplus E_8(-1)^{\oplus 2} \oplus A_2(-1),\]
with 
\[i_1(\gamma_1 - k\mu_2) = -k\mu_2 \text{ and } i_2(\gamma_2) = -k\mu_1\]
and otherwise being given by the identity in the obvious way.
\begin{proposition} \label{prop:twistedpoincarelatticemaps}
    The diagrams 
    \[
    \begin{tikzcd}
    \langle \isocls_1,\isocls_2\rangle^\perp \arrow[d,hook] \arrow[rr,hook,"i_1"] & & \isocls_1^\perp /\isocls_1 \arrow[r,"\cdot \epsilon"] & \isocls_1^\perp/\isocls_1 \arrow[d,hook] \\
    \rHt(Y,\bbQ) \arrow[r,"(\Phi_{\cP_{X,Y}}^{\rHt})^{-1}"] & \rHt(X,-B_X,\bbQ) \arrow[rr,"\cdot \exp(-B_X')"] & & \rHt(X,\bbQ)
    \end{tikzcd}
    \]
    and 
    \[
    \begin{tikzcd}
    \langle \isocls_1,\isocls_2\rangle^\perp \arrow[d,hook] \arrow[rr,hook,"i_2"] & & \isocls_2^\perp/\isocls_2 \arrow[r, "\cdot \epsilon"] & \isocls_2^\perp/\isocls_2 \arrow[d,hook] \\
    \rHt(Y,\bbQ) \arrow[r,"\Phi_{\cP_{Y,Z}}^{\rHt}"] & \rHt(Z,B_Z,\bbQ) \arrow[rr,"\cdot \exp(B_Z')"] & & \rHt(Z,\bbQ)
    \end{tikzcd}
    \]
    commute.
\end{proposition}
\begin{proof}
    We prove the commutativity of the second diagram; the first is similar. Write \[v = a\gamma_2+v_2 \in  \langle \isocls_1, \isocls_2\rangle^\perp\] where $v_2 \in \langle \gamma_1 - k\mu_2\rangle \oplus U \oplus E_8(-1)^{\oplus 2}\oplus A_2(-1)$. We observe that the bottom row of the diagram factors by Proposition \ref{prop:extmukaipoincare} as 
    \[\rHt(Y,\bbQ) \xrightarrow{\rho_{Y,Z}} \rHt(Z,\bbQ) \xrightarrow{\Phi_{\cP_Z}^{\rHt}}\rHt(Z,\bbQ) \xrightarrow{\cdot \exp(B_Z')}\rHt(Z,\bbQ).\]
    In particular, we find, using Proposition \ref{prop:rank4sublattice} and the explicit description of the B-fields in Lemma~\ref{lem:explicitBfields}, that 
    \begin{align*}
        \exp(B_Z')\cdot(\Phi_{\cP_Z}^{\rHt} \circ \rho_{Y,Z})(a\gamma_2+v_2) &= \exp(B_Z') \cdot  \Phi_{\cP_Z}^{\rHt}(a\isocls_Z - ak\mu_1+v_2) \\
        &= \exp(B_Z') \cdot (-af - ak\epsilon\mu_1 + \epsilon v_2) \\
        &= -af-ak\epsilon\mu_1 + \epsilon v_2+ (-ak\epsilon\mu_1+\epsilon v_2,B_Z')f \\
        &= -af - ak\epsilon\mu_1 + \epsilon v_2 + af \\
        & = \epsilon(ak\mu_1 + \epsilon v_2) \\
        &= \epsilon i_2(v),
    \end{align*}
    which allows us to conclude.
\end{proof} 
Since all of the rows in the diagrams of Proposition \ref{prop:twistedpoincarelatticemaps} are isometries, we may invert the top row of the first diagram and compose with the top row of the second diagram to conclude: 
\begin{corollary} \label{cor:kcyclic}
    The rational isometry $\widetilde\Psi_\cQ \colon \rHt(X,\bbQ) \to \rHt(Z,\bbQ)$ restricts to the map 
    \[ \Hlprim^2(X,\bbQ) \simeq (\isocls_1^\perp/\isocls_1)_\bbQ \xrightarrow{i_1^{-1}} \langle \isocls_1,\isocls_2 \rangle_\bbQ^\perp \xrightarrow{i_2} (\isocls_2^\perp/\isocls_2)_\bbQ \simeq \Hlprim^2(Z,\bbQ).\]
    In particular, this isometry is of $k$-cyclic type in the sense of \cite[Definition 3.1]{buskin}.
\end{corollary}
\begin{proof}
    The first claim is clear from the previous Proposition \ref{prop:twistedpoincarelatticemaps}. To check that this isometry is $k$-cyclic, we observe that (with respect to the markings used above) it coincides with the reflection 
    \[v \mapsto v - \frac{2(v,u)}{(u,u)}u\]
    across the vector $u = \gamma_2 + k\mu_2$, up to composition with the integral isometry 
    \[\Hlprim^2(X,\bbZ) \iso  \Hlprim^2(Z,\bbZ)\]
    which is given by $\mu_2 \mapsto \mu_1$, $\gamma_2 \mapsto \gamma_1-k\mu_2$ and the identity on the orthogonal complement which is isomorphic to $U \oplus E_8(-1)^{\oplus 2}\oplus A_2(-1)$ as discussed above.
\end{proof}
\begin{remark}
    While we choose to omit them here due to the presence of additional constants in the results, similar statements to Lemma \ref{lem:explicitBfields} and Corollary \ref{cor:kcyclic} hold (with identical proofs) for Construction \ref{construction:hklattice} for hyperk\"ahler manifolds of $\Kthreen$-type. 
\end{remark}
We can also give a criterion for when an integral vector in the Lagrangian primitive cohomology of $X$ or $Z$ is in the image of $\langle \eta_1,\eta_2\rangle^\perp$; this can be interpreted as the analogue of the formula directly before \cite[Conclusion 3.8]{buskin} for Buskin's cyclic isometries. 
\begin{proposition} \label{prop:divcondition}
    Let $v \in \Hlprim^2(X,\bbZ)$. Then \[k \mid (v,\gamma_2) \Longleftrightarrow v \in \mathrm{im}(i_1).\] 
    Likewise, if $v \in \Hlprim^2(Z,\bbZ)$, then \[k \mid (v,\gamma_1) \Longleftrightarrow v \in \mathrm{im}(i_2).\]
\end{proposition}

\begin{proof}
    This follows from the explicit description of the maps $$i_1 \colon  \langle \eta_1,\eta_2\rangle^\perp \to \eta_1^\perp/\eta_1 \simeq \Hlprim^2(X,\bbZ)$$ and $$i_2 \colon \langle \eta_1,\eta_2\rangle^\perp \to \eta_2^\perp/\eta_2 \simeq \Hlprim^2(Z,\bbZ)$$ given above.
\end{proof}

Corollary \ref{cor:kcyclic} has another interesting consequence. Recall that if $C \subset \mathbb P^5$ is a smooth cubic fourfold, then the Fano variety $F_1(C)$ of lines on $C$ is a hyperkähler manifold of $\mathrm{K}3^{[2]}$-type.
\begin{theorem}\label{thm:fanoequivalence}
    There is a twisted derived equivalence 
    \[\Db(F_1(C_X),\delta_X) \simeq \Db(F_1(C_Z),\delta_Z)\]
    for Brauer classes $\delta_X \in \Br(F_1(C_X))$ and $\delta_Z \in \Br(F_1(C_Z))$.
\end{theorem}
\begin{proof}
    We first recall that by Example \ref{ex:lagprimog10} there is an identification 
    \[\Hlprim^2(X,\bbZ) \simeq \Hprim^4(C_X,\bbZ)(1).\]
    Moreover, it follows from a theorem of Beauville--Donagi \cite{beauvilledonagi}, see also \cite[Corollary 6.3.21]{huybrechtscubics}, that the Fano correspondence induces an isometry 
    \[\Hprim^4(C_X,\bbZ)(1) \simeq \Hprim^2(F_1(C_X,\bbZ)).\]
    Using the isometry of Corollary \ref{cor:kcyclic} and extending by sending the Pl\"ucker polarization $g_X$ to the Pl\"ucker polarization $g_Z$, it follows that there is a rational Hodge isometry 
    \[\rH^2(F_1(C_X),\bbQ) \simeq \Hprim^2(F_1(C_X),\bbQ) \oplus \bbQ g_X\simeq \Hprim^2(F_1(C_Z),\bbQ) \oplus \bbQ g_Z \simeq \rH^2(F_1(C_Z),\bbQ)\]
    which is $k$-cyclic. \par 
    To be more precise, reflection across the same element $u = \gamma_2 + k\mu_2$ induces a rational self-isometry of $\Hprim^2(F_1(C_X),\bbQ) \oplus \bbQ g_X$, and since the integral isometry \[\Hprim^2(F_1(C_X),\bbZ) \to \Hprim^2(F_1(C_Z),\bbZ)\] is the identity on the discriminant group, the induced integral isometry 
    \[\Hprim^2(F_1(C_X),\bbZ) \oplus \bbZ g_X \iso \Hprim^2(F_1(C_Z),\bbZ) \oplus \bbZ g_Z\]
    extends to an isometry of the full second cohomology. Composing the reflection with this integral isometry gives a $k$-cyclic Hodge isometry. \par 
    Since any orientation-preserving integral isometry of $\rH^2(-,\bbZ)$ for hyperk\"ahler manifolds of K3$^{[2]}$-type is a parallel transport operator by \cite[Lemma 9.2]{markmanmonodromy}, it then follows from the existence of such a $k$-cyclic Hodge isometry and \cite[Corollary 8.5]{markmanalgebraic} that there is a hyperholomorphic sheaf inducing a twisted derived equivalence of $F_1(C_X)$ and $F_1(C_Z)$.
\end{proof}
\subsection{The D-equivalence conjecture for \texorpdfstring{hyperk\"ahler}{hyperkähler} manifolds of OG10-type}
We finally present an analogue of the main theorem of \cite{msyz} in the case of the $\OG$-type. 
\begin{theorem}\label{thm:dequivalence}
    Suppose there is a birational equivalence $X \dashrightarrow X'$ between projective hyperk\"ahler manifolds of OG10-type. Then there is a derived equivalence $\Db(X) \simeq \Db(X')$.
\end{theorem}
\begin{remark}
    Optimistically, one should expect that the twisted variant of \cite[Theorem 0.3]{msyz} still holds, and that a similar proof will suffice. As the proof will show, this is related to the question of whether a Brauer class admits a B-field lift which is isotropic with respect to the BBF form.
\end{remark}
We will follow closely the strategy of \cite{msyz}. We will need to make use of the following case of a criterion of Eichler. 
\begin{lemma}[{\cite[Section 10]{eichler} (see also \cite[Proposition 3.3(i)]{ghs})}] \label{lem:eichlercriterion}
    Let $v_1, v_2 \in \Lambda_{\OG}$ be primitive vectors such that $v_1^2 = v_2^2$ and the classes of $\frac{v_1}{\div(v_1)}$ and $\frac{v_2}{\div(v_2)}$ coincide in the discriminant group. Then there is an orientation-preserving isometry $g \in \widetilde{\SO}^+(\Lambda_{\OG})$ acting trivially on the discriminant group such that $gv_1 = v_2$.
\end{lemma}
As input, we first establish a technical lemma for hyperk\"ahler manifolds of $\OG$-type analogous to those established in the $\Kthreen$ case in \cite{msyz}.
\begin{lemma}[{Analogue of \cite[Lemma 2.1 and Proposition 2.2]{msyz}}] \label{lem:og10div1class}
    Suppose $X$ is a projective hyperk\"ahler of $\OG$-type with Picard rank $\geq 2$, and $\cW \in \rH^{1,1}(X,\bbZ)$ is an MBM class. There exists a constant $N' > 0$ such that for any $M > 0$, we can find a class $\cD \in \Pic(X)$ with divisibility 1, $(\cW,\cD) = N'$ and $\cD^2 > M$.
\end{lemma}
\begin{proof}
    We first construct a class $\cA \in \Pic(X)$ of divisibility 1 such that $(\cW,\cA) \neq 0$ as follows. Recall that the discriminant of $\Lambda_{\OG}$ is $\bbZ/3\bbZ$. \par 
    Let $\langle H_1, H_2\rangle$ generate a primitive nondegenerate sublattice of $\Pic(X)$ containing $\cW$ (e.g. the saturation of the sublattice generated by the ample class on $X$ and $\cW$). We may assume that either $\div(H_1) = 1$ or $\div(H_2) = 1$; otherwise since the divisibility divides the order of the discriminant group we have $\div(H_1) = \div(H_2) = 3$, and for some choice of sign it follows that
    \[H_1/3 \pm H_2/3\]
    is trivial in the discriminant group $\bbZ/3\bbZ$, contradicting the primitivity of $\langle H_1, H_2\rangle$. \par 
    Assume without loss of generality that $H_1$ has divisibility 1 and $H_2$ has divisibility 3. If $(\cW,H_1) \neq 0$, we may take $\cA = H_1$. Otherwise, the nondegeneracy forces $(\cW,H_2) \neq 0$; then since $\langle H_1 + H_2, H_2 \rangle = \langle H_1, H_2\rangle$ generates the same primitive sublattice, the assumption that $\div(H_2) = 3$ implies that $H_1 + H_2$ must be of divisibility 1, so we may take $\cA = H_1 + H_2$ and $C_1 = (\cW,\cA)$, which we may assume up to replacing $\cA$ by $-\cA$ is positive. \par 
    We may then proceed as in \cite[Proposition 2.2]{msyz} by taking some class $\omega \in \Pic(X) \cap \cW^\perp$ with $\omega^2 > 0$ and setting $\cD = \cA + t\omega$ for appropriate $t \gg 0$ which has $\div(\cD) = 1$. 
\end{proof}
We also make use of the following. 
\begin{lemma} \label{lem:isotropicmodification}
    Let $\cD \in \Pic(X)$ be a class of divisibility $1$ satisfying $\cD^2 = 2g$. Then there exists an integral vector $v \in \rH^2(X,\bbZ)$ such that $\cD + gv$ is an isotropic vector of divisibility $1$.
\end{lemma}
\begin{proof}
    Since $\cD$ is of divisibility 1, by Lemma \ref{lem:eichlercriterion}, we may write $\cD = \gamma_1' + g\mu_1'$ for $U = \langle \gamma_1',\mu_1'\rangle$ a copy of the hyperbolic plane. Moreover, it is clear that the orthogonal complement of $\langle \gamma_1',\mu_1'\rangle$ in $H^2(X, \mathbb Z)$ contains a second copy $U = \langle \gamma_2',\mu_2'\rangle$ of the hyperbolic plane. \par 
    It is then enough to take 
    \[v = (g-1)\mu_1' + (\gamma_2'-\mu_2'),\]
    as 
    \begin{align*}
        (\cD+gv)^2 &= (\gamma_1'+g^2\mu_1')^2+g^2(\gamma_2'-\mu_2')^2 \\
        &= 2g^2 - 2g^2 = 0
    \end{align*}
    and $(\cD + gv, \mu_1') = 1$.
\end{proof}
For the following, we closely follow the arguments and notation used in \cite[Section 2]{msyz}.
\begin{proof}[Proof of Theorem \ref{thm:dequivalence}]
    If $\rank\Pic(X) = 1$, the result is immediate since $X$ therefore has no nontrivial birational models, so we may assume that $\rank\Pic(X) \geq 2$. Moreover, we may assume that $X$ and $X'$ arise from adjacent chambers in the birational K\"ahler cone of $X$, and are separated by a wall coming from an MBM class $\cW \in \rH^{1,1}(X,\bbZ).$ \par 
    Using Lemma \ref{lem:og10div1class}, there exists a class $\cD \in \Pic(X)$ of divisibility 1 such that \[(\cW,\cD) = N' \qquad \text{and} \qquad \cD^2 = 2g > 2\mbmconst N',\] where $N$ is the constant from Theorem \ref{thm:mbmbound}. By using Lemma \ref{lem:isotropicmodification}, there exists also a class $v \in \rH^2(X,\bbZ)$ such that 
    \[(\cD + gv)^2 = 0\]
    and $\cD+gv$ is of divisibility $1$.
    \par 
    We now carefully choose a parallel transport of $X$ to a compactified LSV fibration. Choose a pair of hyperk\"ahler manifolds of $\OG$-type $X_0, Z_0$ as in Construction \ref{construction:hklattice} with $k = g$; moreover, since we are free to take $g \gg 0$ large, we may assume that Proposition \ref{prop:transverse} holds, and that there exists a twisted hyperholomorphic vector bundle $\cQ_0 \in \Db(X_0\times Z_0,\alpha_0\boxtimes \beta_0).$ Let us also recall that there is a B-field $B_{X_0}'$ for $\alpha_0$ which, under the explicit marking 
    \[\rH^2(X_0,\bbZ) \simeq \langle\isocls_{X_0},\theta_{X_0}\rangle \oplus \langle \gamma_2,\mu_2\rangle \oplus U \oplus E_8(-1)^{\oplus2} \oplus A_2(-1),\]
    is identified with $3\isocls_X - \frac{\epsilon\gamma_2}{g}$ by Lemma \ref{lem:explicitBfields}, where $\isocls_X$ is the pullback of $\cO_{\bbP^5}(1)$ under the Lagrangian fibration and $\gamma_2 \in \Hlprim^2(X_0,\bbZ)$ is an isotropic vector of divisibility 1 orthogonal to $\isocls_X$ and $\theta_X$. 
    \par
    Now observe that the elements 
    \[-\epsilon\gamma_2\qquad \text{and} \qquad \cD + gv \]
    are both isotropic vectors of divisibility 1, so by Lemma \ref{lem:eichlercriterion}, there exists an orientation-preserving isometry 
    \[\rho_{X_0,X} \colon \rH^2(X_0,\bbZ) \iso \rH^2(X,\bbZ)\]
    which sends $-\epsilon \gamma_2$ to $\cD + gv$, which by \cite[Theorem 5.4]{onoratithesis} is given by a parallel transport. \par 
    Fix markings $\marking_{X_0} \colon \rH^2(X_0,\bbZ) \iso \Lambda_{\OG}$ and $\marking_{Z_0} \colon \rH^2(Z_0,\bbZ) \iso \Lambda_{\OG}$, let $\psi \coloneq \widetilde\Psi_{\cQ_0}|_{\rH^2(X,\bbZ)}$ be the associated rational Hodge isometry of $\Lambda_{\OG}$ induced by $\cQ_0$, and let $\fM_{\psi}^0$ be the connected component of the moduli space of pairs containing $(X_0,\marking_{X_0},Z_0,\marking_{Z_0})$. We may complete $X$ and $X'$ to points 
    \[(X,\marking_X,Z,\marking_Z), (X',\marking_{X'}, Z',\marking_{Z'}) \in \fM_\psi^0\]
    such that $(X,\marking_X)$ and $(X',\marking_{X'})$ are inseparable points of $\fM_{\Lambda_{\OG}}$, and where the marking $\marking_{X'}$ is given by composition of $\marking_X$ with the isometry 
    \[\rH^2(X,\bbZ) \iso \rH^{2}(X',\bbZ)\]
    induced by the birational map; it follows also that $(Z,\marking_Z)$ and $(Z',\marking_{Z'})$ are inseparable points of $\fM_{\Lambda_{\OG}}$. \par 
    By Lemma \ref{lem:mbmmap} below, there is a point on the wall associated to the MBM class $\cW \in \Pic(X)$ between $\mathrm{Kah}(X)$ and $\mathrm{Kah}(X')$ which is sent under $\psi$ to the interior of a chamber of $\mathrm{Pos}(X')$, and we may therefore find some $(Z'',\marking_{Z''})$ inseparable from $(Z,\marking_Z)$ giving (using \cite[Lemma 1.2]{msyz}) points in the connected component 
    \[(X,\marking_X, Z'',\marking_{Z''}), (X',\marking_{X'},Z'',\marking_{Z''}) \in \fM_\psi^0.\]
    Moreover, by Theorems \ref{thm:hyperholdef} and \ref{thm:hyperholequivalence} there exist twisted hyperholomorphic vector bundles $\cQ$ on $X \times Y''$ and $X' \times Y''$, with Brauer classes determined using Corollary \ref{cor:bfielddeformationcor} by parallel transport of the B-field 
    \[3\isocls_X - \frac{\epsilon \gamma_2}{gtd}.\]
    Since $\isocls_X$ is integral, it is integral on any deformation and has no contribution to the Brauer class, and we may ignore it. But 
    \begin{align*}
        \rho_{X_0,X}\left(-\frac{\epsilon \gamma_2}{g}\right) &= \frac{\cD + gv}{g} \\
        &= \frac{\cD}{g} + v \in \Pic(X)_\bbQ + \rH^{2}(X,\bbZ)
    \end{align*}
    represents the trivial Brauer class on $X$. Since the marking $\marking_{X'}$ of $X'$ is determined by $\marking_X$ and the birational map (which is an integral Hodge isometry), the Brauer class on $X'$ vanishes as well. Since the Brauer class on $Z''$ is determined completely by the marking, we conclude that there exist derived equivalences 
    \[\Db(X) \simeq \Db(Z'',\alpha_{Z''}) \simeq \Db(X')\]
    for some Brauer class $\alpha_{Z''} \in \Br(Z'').$ This completes the argument. 
\end{proof}
\begin{lemma}[{Analogue of \cite[Proposition 2.4]{msyz}}]\label{lem:mbmmap}
    The rational Hodge isometry 
    \[ \rho_{Z_0,Z}\circ \psi \circ \rho_{X,X_0} \colon \rH^2(X,\bbQ) \iso \rH^2(Z,\bbQ)\]
    does not send $\cW$ to the line spanned by any MBM class on $Z$. 
\end{lemma}
\begin{proof}
    It is enough to show that the image of the parallel transport of $\cW$ to $X_0$ under the isometry 
    \[\psi \colon \rH^2(X_0,\bbQ) \iso \rH^2(Z_0,\bbQ)\]
    is not proportional to any class deformation equivalent to an MBM class on $Z$. \par 
    We proceed by contradiction, so suppose that 
    \[\psi \circ \rho_{X,X_0}(\cW) = \frac{b}{a}\cW'\]
    for $\cW \in \rH^2(Z_0,\bbZ)$ deformation equivalent to an MBM class on $Z$, where $a,b$ are coprime integers. By \cite[Corollary 1.6]{msyz} we know that $a^2 < \mbmconst$. Let us write 
    \[\rho_{X,X_0}(\cW) =\cW_{\lprim} + m_1\isocls_{X_0} + m_2\theta_{X_0}\]
    where $\cW_{\lprim} \in \Hlprim^2(X_0,\bbZ)$. \par 
    Since by Corollary \ref{cor:kcyclic} the class $\psi(\cW_{\lprim}) \in \Hlprim^2(Z_0,\bbQ)$ and $\psi(a\cW) \in \rH^2(Z_0,\bbZ)$, it follows that $\psi(a\cW_{\lprim}) \in \Hlprim^2(Z_0,\bbZ)$, so in particular as this isometry is $g$-cyclic it follows that $a\cW_{\lprim}$ is in the image of the map $i_1$ from the discussion preceding Proposition \ref{prop:twistedpoincarelatticemaps}, and by Proposition \ref{prop:divcondition} we know 
    \[g \mid (-\epsilon\gamma_2,a\cW_{\lprim}) = (-\epsilon \gamma_2,a\rho_{X,X_0}(\cW))\]
    where $\gamma_2 \in \Hlprim^2(X_0,\bbZ)$ as well by the explicit choice of marking on $X_0$. Since parallel transport is an isometry, it follows that 
    \[g \mid (\cD + gv, a\cW)\]
    and hence $g \mid (\cD,a\cW) = aN'$. The rest of the argument is identical to that of \cite[Proposition 2.4]{msyz}.
\end{proof}
\printbibliography
\end{document}